\documentclass[a4paper]{article}

\usepackage{amsmath}
\usepackage{amssymb}
\usepackage{amsthm}
\usepackage{amsfonts}
\usepackage{array}
\usepackage{bbm}
\usepackage{indentfirst}
\usepackage{paralist}
\usepackage{geometry}
\usepackage{color}
\usepackage{multicol}
\usepackage{multirow}
\usepackage{bbm}
\usepackage{graphics}
\usepackage{graphicx}
\usepackage{booktabs}
\usepackage{subfigure}
\usepackage{grffile}
\usepackage{float}
\usepackage{url}
\usepackage[labelsep=quad]{caption}
\usepackage{mathrsfs}
\usepackage{cases}
\usepackage{lscape}
\usepackage{esint}
\usepackage[hidelinks]{hyperref}
\usepackage{multimedia}

\newcommand{\mz}{\mathbf{z}}
\newcommand{\mn}{\mathbf{n}}

\newtheorem{remark}{Remark}
\newtheorem{proposition}{Proposition}
\newtheorem{assumption}{Assumption}
\newtheorem{lemma}{Lemma}
\newtheorem{theorem}{Theorem}[section]

\newtheorem{corollary}{Corollary}
\newtheorem{example}{Example}
\begin{document}

\bibliographystyle{abbrv}

\title{{Regularity of the Value Function in Multi-Asset Portfolio Selection with Proportional Transaction Costs}}
\author{Chen Xinfu\thanks{School of Mathematics, Southwestern University of Finance and Economics. Email: \href{mailto:peisw@swufe.edu.cn}{\nolinkurl{peisw@swufe.edu.cn}}.} \qquad  Qian Shuaijie\thanks{Department of Mathematics, The Hong Kong University of Science and Technology, Hong Kong. Email: \href{mailto:sjqian@ust.hk}{\nolinkurl{sjqian@ust.hk}}.}}
\maketitle
\begin{abstract}
   This paper concerns the variational inequality arising from a specific singular control problem, portfolio selection under proportional transaction costs. This variational inequality is a gradient-constrained partial differential equation (PDE). Generally, the literature only guarantees the $W^2_\infty$ regularity of the solution to such PDEs and counterexamples exist for higher regularity.

   In this paper, by exploiting the concavity of the value function, we
   connect this gradient-constrained problem to an obstacle problem, and  finally show that the solution is $C^{2, \alpha}$ except at the corner points where multiple gradient constraints bind. In particular, if there are two assets, then a global $C^{2, \alpha}$ solution can be guaranteed.

\end{abstract}

\section{Introduction}
In this paper, we are concerned with the following variational inequality from the long-term portfolio choice problem with $n$ stocks and proportional transaction costs.
\begin{align}\label{equ PDE u}
& \min \bigg\{\theta- \mathcal{L} u + \mathcal{A}(u),
\quad  \min\limits_{1\leq i\leq n} \bigg\{u_{z_i}- (1-\mu_i)  \bigg\}, \quad  \min\limits_{1\leq i\leq n} \bigg\{ (1+\lambda_i) -u_{z_i} \bigg\}  \bigg\} = 0,\ \forall \mz=(z_1,\ldots,z_n)\in \mathbb{R}^n,
\end{align}
where $(\sigma_{ij})_{1\leq i,j\leq n}$ is the positive-definite covariance matrix of the stock returns, and
\begin{align}
&\mathcal{L}u = \frac{1}{2}\sum_{i,j=1}^n\sigma_{ij}z_i z_j u_{z_i z_j}+\sum_{i=1}^n(\alpha_i-r)z_i u_{z_i}  \\
& \mathcal{A}(u) = \frac{\gamma}{2}\sum_{i,j=1}^n\sigma_{ij}z_i z_j u_{z_i}u_{z_j}.
\end{align}
We assume the no-trading region is contained in $[1/M,M]^n$ for some $M>1$, as in Assumption~\ref{ass positive no trading}. For a concave solution satisfying the growth condition specified in Section~\ref{Sec model}, we prove that the free boundary is locally Lipschitz and $u$ is locally $C^{2,\alpha}$ at every non-corner free-boundary point, where only one gradient constraint binds, for every $\alpha\in(0,1)$. When $n=2$, we also treat the corner points and obtain $C^{2,\alpha}$ regularity throughout the solvency region.

Generally speaking, the optimal regularity for solutions to gradient-constrained variational inequalities is $W^{2, \infty}$, rather than $C^{2, \alpha}$, as we can see from the following example.
\begin{example}
The function
\begin{align} u(x)=
\begin{cases}
   \frac{1}{2} x^2, & \text{ if }x\in [0, 1)\\
   x-\frac{1}{2}, & \text{ if } x\in [1, 2],
\end{cases}
\end{align}
is a $W^{2, \infty}$ solution to the variational inequality
\begin{align}
    \min\Big\{ -u_{xx}+ 1,\ 1-u_x     \Big\} = 0
\end{align}
with the boundary conditions
\begin{align}
u(0) = 0, \quad u_x(2) = 1.
\end{align}
\end{example}
The $W^{2, \infty}$ regularity for such gradient-constrained variational inequality has already been established by Evans \cite{Evans1979}, Wiegner \cite{Wiegner1981}, and
Hynd \cite{hynd2012eigenvalue, Hynd2013}.

However, there are plenty of examples where the solution is $C^{2, \alpha}$. First, in the one-dimensional economic and financial problems, researchers always use the super-contact condition to find the solution, which matches the second-order derivatives at the free boundary (see Dumas \cite{Dumas1991}). Second, the literature on proportional transaction cost problems with one risky asset already proves the solution to be $C^{2, \alpha}$, see (Dai et al. \cite{daisiam} and Dai and Yi \cite{dai2009finite}).  Lastly, Shreve and Soner \cite{ShreveSoner1991}, Soner and Shreve \cite{SonerShreve1989}, and Soner, Shreve, and El Karoui \cite{SonerShreveElKaroui1991} concern special cases with a very special structure of the second-order operator, and prove the $C^{2, \alpha}$ regularity.

\subsection{Main Proof Idea}
The main proof idea is inspired by Dai et al. \cite{daisiam} and Dai and Yi \cite{dai2009finite} on the one-asset problem, where they transform the gradient-constrained equation for $u$ into an obstacle problem for $u_z$. Regularity of $u_z$ can then be transferred to $u$ by integration in $z$. This idea was also introduced in Shreve and Soner \cite{ShreveSoner1991} and Soner and Shreve \cite{SonerShreve1989}.

In the multi-asset case, the main challenge for this technique is that the obstacle problem for one derivative $u_{z_i}$ will also involve derivatives in the other directions $u_{z_j}$, $j\neq i$. A direct global transformation therefore leads to a coupled system of obstacle problems for the first derivatives.

In this paper, we resort to the following two-step approach to fix this issue. First, starting from the concave viscosity solution with the growth condition specified in Section~\ref{Sec model}, we use penalization and comparison to establish local $W^{2,p}$ regularity for every $1<p<\infty$ under Assumption~\ref{ass positive no trading}.
Second, we treat the first-order derivatives as known functions and look for the connection between the gradient-constrained problem and an obstacle problem in a neighborhood of a non-corner free-boundary point $P$, where only one gradient constraint binds. We construct an auxiliary obstacle problem for the active derivative and use comparison to identify its solution with $u_{z_i}$. This establishes the required equivalence locally.

With this equivalence established, we next prove the Lipschitz continuity of the free boundary near non-corner points and apply Schauder estimates to obtain local $C^{2,\alpha}$ regularity of $u$ for every $\alpha\in(0,1)$. These arguments apply in arbitrary dimension. The local analysis also shows that $u$ cannot be $C^3$ in a neighborhood of a non-corner free-boundary point.

It remains to consider corner points, where several gradient constraints bind. For $n=2$, a comparison argument gives Lipschitz continuity of the free boundaries up to the corners. The affine structure of the trading regions then yields a uniformly elliptic equation with a Lipschitz right-hand side across the interfaces, so Schauder estimates give $C^{2,\alpha}$ regularity at the corners. The corresponding corner analysis in higher dimensions is not covered by the present argument.

\subsection{Related literature on transaction cost problem}

Merton (\cite{merton1969lifetime}, \cite{merton1975optimum}) pioneered continuous-time portfolio selection.
To capture the trading fees or general market illiquidity,
Magill and Constantinides (\cite{magill1976portfolio}) later introduced proportional transaction costs to Merton’s model and showed that a no-transaction region exists. The trading occurs only when the state variable hits the boundary of this region, i.e., the trading boundary.

Much of the classical literature focuses on the case with only one risky asset. Among these works, Shreve and Soner \cite{shreve1994optimal} consider an infinite-horizon investment and consumption problem. Further regularity results are obtained by Dai and Yi \cite{dai2009finite}, who consider the case with one risky asset and a finite investment horizon, so time is an additional variable. This technique is later extended by Dai et al. \cite{daisiam}, where consumption is included.
The analysis of the transaction cost problem has been further developed in a large literature, often relying on the notion of a viscosity solution of the HJB equation.

For multiple risky assets, Possama\"{\i}, Soner, and Touzi \cite{possamai2015homogenization} develop small-transaction-cost asymptotics using homogenization and viscosity-solution methods, and establish $C^{1,1}$ regularity for the first corrector. This result concerns the limiting corrector equation; our analysis concerns regularity across the free boundaries of the original gradient-constrained problem.

The model considered in this paper is from Guasoni and M\"uhle-Karbe \cite{guasoni2015long} and Chen et al. \cite{chen2022asymptotic}, and leads to an eigenvalue problem in the form of a variational inequality with gradient constraints. Guasoni and M\"uhle-Karbe \cite{guasoni2015long} analyze the case with one risky asset or multiple uncorrelated risky assets, while Chen et al. \cite{chen2022asymptotic} treat two illiquid and correlated assets and establish existence of a viscosity solution, unique up to an additive constant, with $C^1$ regularity away from the coordinate axes. Under the stated assumptions, we improve the regularity to $C^{2,\alpha}$ near non-corner free-boundary points in arbitrary dimension and also treat the corners when $n=2$.

The rest of the paper is organized as follows. Section \ref{Sec model} introduces the model and the properties of the solution used in our analysis. Then Section \ref{sec Strong} establishes local $W^{2,p}$ regularity for every $1<p<\infty$. After that, Section~\ref{sec OP and FB} develops the local obstacle reduction and proves Lipschitz continuity of the free boundaries, $C^{2,\alpha}$ regularity, and failure of $C^3$ regularity at non-corner points in general dimension. Section \ref{sec classical} establishes Lipschitz continuity of free boundaries up to the corners and completes the $C^{2,\alpha}$ regularity proof in two dimensions. 

\section{Model}\label{Sec model}
We consider a market with one risk-free asset with return rate $r \geq 0$, and $n$ stocks with price dynamics
\begin{align}
d S^i_t/S^i_t  = \alpha_i dt + \sum_{j = 1}^n  a_{ij} d B_t^j,\ 1\leq i \leq n,
\end{align}
where the constant return rates satisfy $\alpha_i >r$, $(a_{ij})_{1\leq i, j\leq n}$ is a positive-definite matrix, and $\{B_t^j\}_{1\leq j \leq n}$ is a standard $n-$dimensional Brownian motion.

An investor selects a portfolio under proportional transaction costs. The balances in cash (denoted as $X_t$) and stock $i$ (denoted as $Y^i_t$) evolve as
\begin{align}
\begin{cases}
{d X_t = r X_t\,dt - \sum_{i = 1}^n (1+\lambda_i) d L_t^i + \sum_{i = 1}^n (1-\mu_i) d M_t^i}\\
d Y^i_t = Y^i_t \big(\alpha_i  dt +  \sum_{j = 1}^n  a_{ij} d B_t^j\big) + d L_t^i - d M_t^i, \quad 1\leq i, j \leq n\,
\end{cases}
\end{align}
where $\mu_i\in [0, 1)$ and $\lambda_i\in [0, \infty)$ with $\mu_i+\lambda_i>0$ are the transaction cost rates and $ L_t^i$ and $ M_t^i$ stand  for the cumulative purchase and sale amount of stock $i$, respectively.

Define the wealth process as $$W_t := X_t + \sum_{i = 1}^n (1-\mu_i) (Y^i_t)^+ -  \sum_{i = 1}^n (1+\lambda_i) (Y^i_t)^-. $$  The investor considers the following objective
\begin{align}
\min \limits_{L^i_t, M^i_t} \limsup \limits_{T\to \infty} \frac{1}{T}\ln  E\bigg[  \exp\Big(-\gamma \{W_T- (x+\sum_{i = 1}^n y_i) e^{rT}\}\Big) \bigg| X_0 = x, Y^i_0 = y^i, 1\leq i \leq n\bigg],
\end{align}
for some $\gamma >0$.

We see that (see \cite{chen2013characterization} and \cite{chen2022asymptotic}) the above control problem is related to the following HJB equation.
\begin{align}\label{equ PDE u}
& \min \bigg\{\theta- \mathcal{L} u + \mathcal{A}(u),
\quad  \min\limits_i \bigg\{u_{z_i}- (1-\mu_i)  \bigg\}, \quad  \min\limits_i \bigg\{ (1+\lambda_i) -u_{z_i} \bigg\}  \bigg\} = 0,\ \forall \mathbf{z} \in   \mathbb{R}^n
\end{align}
where $\mathbf{z} = (z_i)_{1\leq i \leq n}$, $\sigma_{ij} = \sum_{k = 1}^n a_{ik} a_{jk}$, and
\begin{align}
&\mathcal{L}u = \frac{1}{2} \sum_{1\leq i, j\leq n} \sigma_{ij} z_i z_j u_{z_i z_j} + \sum_{i = 1}^n (\alpha_i -r) z_i u_{z_i}  \\
& \mathcal{A}(u) = \frac{\gamma}{2} \sum_{i, j= 1}^n \sigma_{ij} z_i z_j u_{z_i} u_{z_j}.
\end{align}

Denote $\ell(\mathbf {z}) =\sum_{i=1}^n \ell_i\left(z_i\right)$, with  $$\quad \ell_i\left(z_i\right)= \begin{cases}\left(1-\mu_i\right) z_i & \text { if } z_i \geq 0 \\ \left(1+\lambda_i\right) z_i & \text { if } z_i<0\end{cases}.$$

Based on Hynd \cite{hynd2012eigenvalue, Hynd2013}, \cite{chen2013characterization} and \cite{chen2022asymptotic} has the following result.
\begin{proposition}\label{prop u MF}
    Problem \eqref{equ PDE u} has a unique viscosity solution $(\theta, u)$ with the growth condition $\lim \limits_{|\mz| \rightarrow \infty} u(\mathbf{z}) / \ell(\mz)=1$. Moreover, $u$ is concave.
    In addition, $u\in C^1((\mathbb R\setminus\{0\})^n)$.
\end{proposition}
In the arguments below, we only use $u(\mz)-\ell(\mz)=o(|\mz|)$ as $|\mz|\to\infty$, which follows from the stated growth condition.

Define the no-trading region $\mathbf{N}_i$, buy region $\mathbf{B}_i$, and sell region $\mathbf{S}_i$ for risky asset $i$:
$$
\begin{gathered}
\mathbf{N}_i=\left\{\mz \in \mathbb{R}^n \mid 1-\mu_i<u_{z_i}<1+\lambda_i\right\}, \\
\mathbf{B}_i=\left\{\mz \in \mathbb{R}^n \mid u_{z_i}=1+\lambda_i\right\}, \quad \mathbf{S}_i=\left\{\mz \in \mathbb{R}^n \mid u_{z_i}=1-\mu_i\right\}
\end{gathered}
$$
Then from the concavity of $u$, there are  measurable functions $l_i^{ \pm}: \mathbb{R}^{n-1} \mapsto \mathbb{R}\cup \pm \infty$ such that
$$
\begin{gathered}
\mathbf{N}_i=\left\{\mz \in \mathbb{R}^n \mid z_i \in \big(l_i^-(z_{-i}), l_i^+(z_{-i}) \big) \right\}, \\
\mathbf{B}_i=\left\{\mz \in \mathbb{R}^n \mid z_i \leq l_i^-(z_{-i}) \right\}, \quad \mathbf{S}_i=\left\{\mz \in \mathbb{R}^n \mid z_i \geq l_i^+(z_{-i}) \right\}
\end{gathered}
$$
 From the concavity of the function $u$, we have the following important proposition.
\begin{proposition}
If $u_{z_i} = 1-\mu_i$, for $i = i_1, i_2, ..., i_k$ at  some point ${\bf z} = (z_i)_{1\leq i \leq n}$, then for any ${\bf \hat z} = (\hat z_i)_{1\leq i \leq n}$ such that $\hat z_i = z_i$, $\forall i \neq i_1, i_2, ..., i_k$ and $\hat z_i \geq  z_i$, $\forall i = i_1, i_2, ..., i_k$, we have
$u_{z_i} = 1-\mu_i$, for $i = i_1, i_2, ..., i_k$ at this point.
\end{proposition}

\begin{lemma}\label{lem bounded free boundaries}
The free boundaries $l_i^\pm$ are uniformly bounded. More precisely, there is $M>0$, independent of $z_{-i}$, such that
\begin{align}\label{equ bounded free boundaries}
    -M\leq l_i^-(z_{-i})\leq l_i^+(z_{-i})\leq M,
    \qquad z_{-i}\in\mathbb{R}^{n-1},\quad 1\leq i\leq n.
\end{align}
\end{lemma}
\begin{proof}
The growth condition implies $u(\mz)-\ell(\mz)=o(|\mz|)$: use $|\ell(\mz)|\leq C|\mz|$ where $\ell\neq0$, and extend by continuity to $\ell=0$. Let $K=\prod_{i=1}^n[1-\mu_i,1+\lambda_i]$ and define
\[
    H(p)=\sup_{\mz\in\mathbb{R}^n}\{u(\mz)-p\cdot\mz\},\qquad p\in K.
\]
\noindent\textbf{(i) Interior contact points.} For $p\in\operatorname{int}K$, the identity $\ell(\mz)=\min_{q\in K}q\cdot\mz$ gives $\ell(\mz)-p\cdot\mz\leq-c_p|\mz|$ for some $c_p>0$. Hence the supremum is attained at a point $\mz^p$. The affine upper test $p\cdot\mz+H(p)$ has all gradient constraints strict, so the viscosity subsolution property yields
\begin{align}\label{equ affine contact bound}
    \theta-\sum_{i=1}^n(\alpha_i-r)z_i^p p_i
    +\frac{\gamma}{2}\sum_{i,j=1}^n\sigma_{ij}z_i^p z_j^p p_i p_j\leq0.
\end{align}
Positive definiteness of $(\sigma_{ij})$ and $p_i\geq\min_j(1-\mu_j)>0$ imply $|\mz^p|\leq M$, uniformly over all interior slopes and their contact points.

\noindent\textbf{(ii) Lipschitz continuity on $K$.} Thus $H(p)=\max_{|\mz|\leq M}\{u(\mz)-p\cdot\mz\}$ for interior $p$. This identity also holds on $\partial K$: approximate $p$ by interior $p^k$, pass to the limit in
\[
    u(\mz)-p^k\cdot\mz\leq
    \max_{|w|\leq M}\{u(w)-p^k\cdot w\}
\]
for each fixed $\mz$, and then take the supremum over $\mz$. Consequently,
\begin{align}\label{equ dual Lipschitz boundedness}
    |H(p)-H(p')|\leq M|p-p'|,\qquad p,p'\in K.
\end{align}

\noindent\textbf{(iii) Uniform bound on the free boundaries.} At any $\mz$, concavity provides a supporting slope $p$, meaning $u(w)\leq u(\mz)+p\cdot(w-\mz)$ for all $w$. The growth condition along $w=\mz\pm te_i$ gives $p\in K$, and $H(p)=u(\mz)-p\cdot\mz$. If $z_i>M$, replace $p_i$ by $1-\mu_i$, keeping the other components fixed, to obtain $p'\in K$. Then
\[
    \big(p_i-(1-\mu_i)\big)z_i
    \leq H(p')-H(p)
    \leq M\big(p_i-(1-\mu_i)\big),
\]
which forces $p_i=1-\mu_i$. Similarly, replacing $p_i$ by $1+\lambda_i$ when $z_i<-M$ forces $p_i=1+\lambda_i$. These conclusions hold for every supporting slope, hence $u_{z_i}=1-\mu_i$ when $z_i>M$ and $u_{z_i}=1+\lambda_i$ when $z_i<-M$. By the definitions of the trading regions, $l_i^+(z_{-i})\leq M$ and $l_i^-(z_{-i})\geq-M$. Together with $l_i^-\leq l_i^+$, this proves \eqref{equ bounded free boundaries}.
\end{proof}

\begin{lemma}\label{lem weighted hessian bound}

The function $u$ belongs to $C^{1,1}_{\mathrm{loc}}((\mathbb R\setminus\{0\})^n)$, and there is a constant $C>0$ such that
\[
    \max_{1\leq i,j\leq n}
    \bigl\|z_i z_j\,u_{z_i z_j}\bigr\|_{L^\infty((\mathbb R\setminus\{0\})^n)}
    \leq C.
\]

\end{lemma}
\begin{proof}

\noindent\textbf{(i)} For any subset $I\subset\{1,\ldots,n\}$, we define $J=I^c$. Consider points $\mz$ in the interior of the set
$$\{\mz | u_{z_i}(\mz)\in\{1-\mu_i,1+\lambda_i\} \forall i\in I \text{ and } u_{z_j} \in (1-\mu_j,1+\lambda_j) \forall j \in J  \}. $$
We justify the $C^{1, 1}$ bound for those points first.

Define
\[
    v_{I,p}(y):=\sup_{x_I\in\mathbb R^I}
        \left\{u(x_I,y)-\sum_{i\in I}p_i x_i\right\},
    \qquad y\in\mathbb R^J.
\]
By Lemma \ref{lem bounded free boundaries}, the supremum equals evaluation at $x_i=M$ if $p_i=1-\mu_i$ and $x_i=-M$ if $p_i=1+\lambda_i$.
Thus $v_{I,p}(y)$ is obtained by fixing the coordinates $x_i$, $i\in I$, at the values specified above and subtracting the constant $\sum_{i\in I}p_i x_i$. It is concave because $u$ is concave. Since the fixed coordinates are nonzero, Proposition \ref{prop u MF} gives $v_{I,p}\in C^1((\mathbb R\setminus\{0\})^J)$. Moreover, with $x_I$ fixed as above, for every $y\in(\mathbb R\setminus\{0\})^J$ and $j\in J$,
\[
    \partial_{y_j}v_{I,p}(y)=u_{z_j}(x_I,y),
    \qquad
    1-\mu_j\leq\partial_{y_j}v_{I,p}(y)\leq1+\lambda_j.
\]

Choose $p^\delta\in\prod_{i\in I}(1-\mu_i,1+\lambda_i)$ with $p^\delta\to p$ as $\delta\to0$, and we have
\[
    v_{I,p^\delta}(y):=\sup_{x_I\in\mathbb R^I}
       \left\{u(x_I,y)-\sum_{i\in I}p_i^\delta x_i\right\},
    \qquad y\in\mathbb R^J.
\]
Its maximizers lie in $[-M,M]^I$, and
\[
    |v_{I,p^\delta}-v_{I,p}|
    \leq M\sum_{i\in I}|p_i^\delta-p_i|.
\]
If $\psi$ is a $C^2$ test function touching $v_{I,p^\delta}$ from above at $y$, then
$\sum_{i\in I}p_i^\delta x_i+\psi(y)$ is an upper test for $u$ at a maximizing point.
All the original gradient constraints are strict there, so every coordinate of this contact point lies in $[-M,M]$ by Lemma \ref{lem bounded free boundaries}.
The viscosity subsolution inequality and $\mathcal A(u)\geq0$ therefore give
\[
    -\operatorname{Tr}\bigl(A_J(y)D^2\psi(y)\bigr)\leq C_0,
    \qquad
    A_J(y)=\frac12\operatorname{diag}(y)\Sigma_{JJ}
                         \operatorname{diag}(y),
\]
with $C_0>0$ independent of $\delta$, $I$, $p$, and the contact point.
Viscosity stability gives the same inequality for $v_{I,p}$ on its strict-gradient region (see \cite[Proposition~4.3 and Lemma~6.1]{CrandallIshiiLions1992}).
For any center $y^0$ with nonzero coordinates, set $y_j=y_j^0\xi_j$ and $\widetilde v(\xi)=v_{I,p}(y)$ on $(1/2,3/2)^J$.
The second-order matrix in these coordinates is
\[
    \widetilde A_J(\xi)
    =\frac12\operatorname{diag}(\xi)\Sigma_{JJ}
                         \operatorname{diag}(\xi)
    \succeq\frac{\kappa}{8}I,
    \qquad \kappa=\lambda_{\min}(\Sigma)>0.
\]
Therefore, the positive matrix-valued measure $\nu=-D_\xi^2\widetilde v$ satisfies, on the strict-gradient region,
\[
    \frac{\kappa}{8}\operatorname{tr}\nu
    \leq\operatorname{tr}(\widetilde A_J\nu)\leq C_0\,d\xi.
\]
Consequently $\nabla_\xi\widetilde v$ has the local Lipschitz bound $L:=8C_0/\kappa$, uniformly in the center and the choice of $I,p$. Since $\nabla_I u$ is constant, this already implies the $C^{1,1}$ bound for $u$.

\noindent\textbf{(ii)}
By (i), after rescaling, the gradient of $u$ has the same local Lipschitz bound $L$ on each piece with a fixed binding pattern, since the binding gradient components are constant. As $u\in C^1$ and there are only finitely many binding patterns, the gradient is continuous across their junctions, and these uniform bounds combine to give the Lipschitz bound $L$ throughout each rescaled box. Hence $u\in C^{1,1}_{\mathrm{loc}}((\mathbb R\setminus\{0\})^n)$, and undoing the rescaling yields the asserted weighted Hessian bound with $C=9L/4$.

\end{proof}

\begin{figure}[htbp]
\begin{minipage}{\linewidth}\centering		\includegraphics[width=0.8\textwidth]{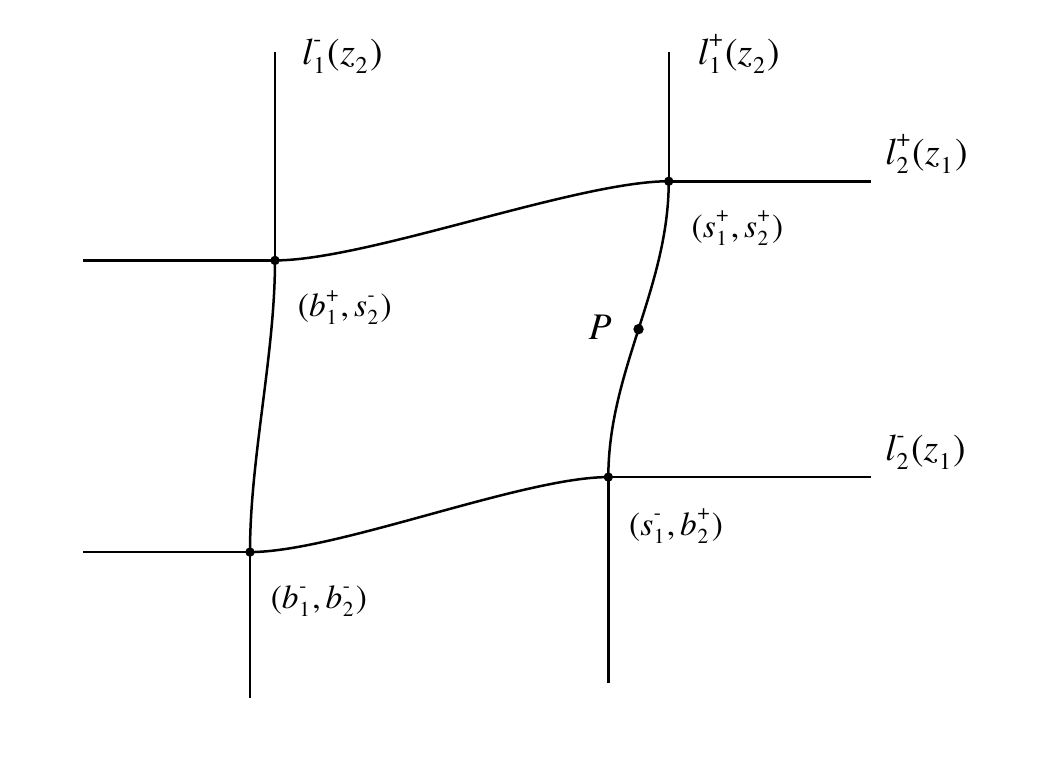}
\end{minipage}
\caption{Illustrative figure of the free boundary}
\label{fig:free-boundary}
\end{figure}

\begin{corollary}\label{coro semiconti fb}
For $1\leq i\leq n$, the free boundaries $l_i^+$ are lower semicontinuous and $l_i^-$ are upper semicontinuous on $[-M,M]^{n-1}$.
\end{corollary}
\begin{proof}[Proof of Corollary \ref{coro semiconti fb}]

Since $u$ is finite and concave on $\mathbb R^n$, it is continuous. It suffices to consider $i=1$. Let $w^k\in[-M,M]^{n-1}$ converge to $w^0$.

\noindent\textbf{(i) Lower semicontinuity of $l_1^+$.} By Lemma \ref{lem bounded free boundaries}, we can choose a subsequence such that
\[
 l_1^+(w^k)\longrightarrow s:=\liminf_{k\to\infty}l_1^+(w^k).
\]
Fix any $s<a<b$. For all sufficiently large $k$, the interval $[a,b]$ lies in the sell tail, so
\[
 u(b,w^k)-u(a,w^k)=(1-\mu_1)(b-a).
\]
Continuity of $u$ gives the same identity at $w^0$. Since $s<a<b$ are arbitrary, $u(\cdot,w^0)$ is affine with slope $1-\mu_1$ on $(s,\infty)$. Hence
\[
 l_1^+(w^0)\leq s=\liminf_{k\to\infty}l_1^+(w^k).
\]

\noindent\textbf{(ii) Upper semicontinuity of $l_1^-$.} Analogously, using the buy tail with slope $1+\lambda_1$ and a subsequence realizing the limsup proves upper semicontinuity of $l_1^-$.
The same argument applies to every coordinate $i$.

\end{proof}

\section{Theoretical Analysis: Strong solution}\label{sec Strong}

\begin{assumption}\label{ass positive no trading}
Throughout this section, we assume that there is a constant $M>1$ such that
\[
    \mathbf N:=\bigcap_{i=1}^n\mathbf N_i\subset[1/M,M]^n.
\]
\end{assumption}

In this subsection, we start with the construction of a strong solution
of the PDE problem. Then we verify that this solution coincides with the function $u$ under consideration.

Throughout this section, $\mz=(z_1,\ldots,z_n)$, all component indices range from $1$ to $n$, and $\mathbf N:=\bigcap_{i=1}^n\mathbf N_i$ is the no-trading region. Define
\begin{align}
F^v(\mz,t)=\theta(T-t)+u(\mz).
\end{align}

Then $F^v$ is a viscosity solution to the following PDE problem.

\begin{numcases}{}
\min\bigg\{-F_t-\mathcal L F+\mathcal A(F),
\min_{1\leq i\leq n}\{F_i-(1-\mu_i)\},\notag\\
\qquad\qquad\min_{1\leq i\leq n}\{(1+\lambda_i)-F_i\}\bigg\}=0,
\quad(\mz,t)\in(1/M,M)^n\times(0,T)\label{equ PDE F}\\
F(\mz,T)=u(\mz),\qquad \mz\in[1/M,M]^n\label{equ PDE F bd1}\\
F(\mz,t)=F^v(\mz,t),\quad z_i=1/M,\quad \mz_{-i}\in[1/M,M]^{n-1},\ t\in[0,T)\\
F(\mz,t)=F^v(\mz,t),\quad z_i=M,\quad \mz_{-i}\in[1/M,M]^{n-1},\ t\in[0,T).\label{equ PDE F bd3}
\end{numcases}

\begin{theorem}\label{thm F eta}

Under Assumption \ref{ass positive no trading}, the solution $F=F^v$ of \eqref{equ PDE F} with the terminal and boundary conditions \eqref{equ PDE F bd1}--\eqref{equ PDE F bd3} satisfies, for every $1<p<\infty$,
\[
F\in W^{2,1}_{p,\mathrm{loc}}((1/M,M)^n\times(0,T)),
\qquad u\in W^2_{p,\mathrm{loc}}((1/M,M)^n).
\]
Moreover, $u\in C^\infty(\mathbf N)$ and there is a constant $C<\infty$ such that
\[
\sup_{\mz\in\mathbf N}\|D^2u(\mz)\|\leq C.
\]

\end{theorem}

\begin{proof}[\bf Proof of Theorem \ref{thm F eta}]

{\bf Step 1.}
We first use the penalty method to construct a solution $F\in W^{2,1}_{p,\mathrm{loc}}((1/M,M)^n\times(0,T))$ for every $1<p<\infty$. The Sobolev estimates in this step are taken on cylinders whose closures are contained in $(1/M,M)^n\times(0,T)$, where the second-order operator is uniformly parabolic.

Define the function $\beta^\epsilon(x) = \beta(\frac{x}{\epsilon})$ with the function $\beta(x)$ chosen such that\\
1. It is globally smooth and increasing with an upper bound.\\
2. $\beta(x)=0$ for $x\leq0$, and $\beta(1)=c_\beta$.\\
3. $\beta''(x)\geq 0$, $\forall x \leq 0$.

Consider the penalized equation
\begin{align}\label{equ PDE penal}
 F_t+ & \mathcal L F - \mathcal{A}(F)+ \sum_{i = 1}^{n} (z_i+2M) \beta^\epsilon \bigg(1-\mu_i - F_i  \bigg) \notag \\
 & + \sum_{i = 1}^{n} (-z_i+2M) \beta^\epsilon \bigg( F_i -(1+\lambda_i ) \bigg) = 0, \quad \text{in $ (1/M,M)^n \times (0, T)$}.
\end{align}
\begin{lemma}\label{lem strong solu}

For $0<\epsilon\leq\min_{1\leq i\leq n}(\mu_i+\lambda_i)$ and sufficiently large $c_\beta$ independent of $\epsilon$, the problem \eqref{equ PDE penal} with boundary and terminal conditions \eqref{equ PDE F bd1}--\eqref{equ PDE F bd3} admits a solution $F^\epsilon\in C^\infty((1/M,M)^n\times(0,T))\cap C([1/M,M]^n\times[0,T])$ whose spatial gradient extends continuously to $[1/M,M]^n\times(0,T]$, with $\nabla F^\epsilon(\mz,T)=\nabla u(\mz)$, such that
\begin{align}\label{equ bound nabla Fep}
1-\mu_i-\epsilon\leq F_i^\epsilon\leq1+\lambda_i+\epsilon,\qquad 1\leq i\leq n.
\end{align}

\end{lemma}

\begin{proof}[\bf Proof of Lemma \ref{lem strong solu}]

The interior regularity is straightforward by bootstrapping the interior Schauder estimates \cite[Theorem 4.9]{lieberman1996second}.

\noindent\textbf{Step 1. Boundary $C^1$ regularity.}

\noindent\emph{(i) Reduction to a heat equation.}
Let $P^\epsilon$ denote the sum of the penalty terms in \eqref{equ PDE penal}. Set $\Sigma=(\sigma_{ij})$ and
\[
\mathbf x=\sqrt{2}\,\Sigma^{-1/2}(\log z_1,\ldots,\log z_n)^\top,
\qquad
b=\sqrt{2}\,\Sigma^{-1/2}
\bigl(\alpha_i-r-\tfrac12\sigma_{ii}\bigr)_{i=1}^n.
\]
On the convex parallelepiped $D=\sqrt{2}\,\Sigma^{-1/2}(-\log M,\log M)^n$, define the remainder
\[
w(\mathbf x,t)=e^{b\cdot\mathbf x/2+|b|^2(T-t)/4}
\left(e^{-\gamma F^\epsilon(\mz,t)}-e^{-\gamma F^v(\mz,t)}\right),
\]
where $\mz$ is the inverse image of $\mathbf x$. Direct substitution into \eqref{equ PDE penal} gives, almost everywhere,
\[
\begin{cases}
w_t+\Delta_{\mathbf x}w=h(\mathbf x,t),
&\text{in }D\times(t_\epsilon,T),\\
w=0,&\text{on }\partial D\times[t_\epsilon,T],\\
w(\mathbf x,T)=0,&\text{for }\mathbf x\in\overline D,
\end{cases}
\]
where
\[
h(\mathbf x,t)=\gamma e^{b\cdot\mathbf x/2+|b|^2(T-t)/4}
\left[e^{-\gamma F^\epsilon}P^\epsilon
-e^{-\gamma F^v}\bigl(\theta-\mathcal Lu+\mathcal A(u)\bigr)\right].
\]
The functions on the right are evaluated at $(\mz,t)$. Define
\[
H:=\|h\|_{L^\infty(D\times(t_\epsilon,T))}<\infty.
\]
Finiteness follows from Lemma \ref{lem weighted hessian bound}, boundedness of $\beta^\epsilon$, and continuity of $F^\epsilon$ on the closed cylinder. Constants below may depend on the fixed $\epsilon$.

\noindent\emph{(ii) Continuity up to the interior of boundaries.}
Standard parabolic regularity \cite{lieberman1996second,LadyzhenskayaSolonnikovUraltseva1968} gives continuity of $\nabla_{\mathbf x}w$ up to each open face of $D$ for $t_\epsilon<t<T$.

\noindent\emph{(iii) Continuity up to terminal time.}
The standard heat-kernel and gradient estimates \cite{IvriiPDE,lieberman1996second} give
\[
|w(\mathbf x,t)|\leq H(T-t),
\qquad
|\nabla_{\mathbf x}w(\mathbf x,t)|\leq CH\sqrt{T-t}.
\]
Thus both $w$ and $\nabla_{\mathbf x}w$ converge uniformly to zero as $t\uparrow T$. Undoing the transformations in (i) therefore gives
\[
\nabla_{\mz}F^\epsilon(\mz,t)\longrightarrow\nabla u(\mz)
\qquad\text{uniformly as }t\uparrow T.
\]

\noindent\emph{(iv) Edges and corners.}
We first bound the function $w$ and then establish the continuity of $\nabla w$.

Two adjacent faces enclose a wedge of angle $\omega<\pi$. In its normal plane, use polar coordinates $(r,\varphi)$ with $0\leq\varphi\leq\omega$, and choose $1<q<\min\{2,\pi/\omega\}$. The barrier
\[
C r^q\cos\!\bigl(q(\varphi-\omega/2)\bigr)-\frac H4r^2
\]
has negative Laplacian $H$ and is nonnegative on the bounded domain for sufficiently large $C$, since the cosine is strictly positive on the closed wedge. Comparison bounds $|w| \leq Cr^q$ near the corner.

Fix $(\mathbf x_0,t_0)\in D\times(t_\epsilon,T)$ and let $\delta$ be the second-smallest distance from $\mathbf x_0$ to the supporting faces. Near their intersections, the two closest faces are adjacent and their intersection is at distance at most $C\delta$. Choose $\rho=c\delta$ with $c>0$ sufficiently small, so that $B_\rho(\mathbf x_0)$ meets at most one face. For $t_0$ away from $T$ and $\delta$ sufficiently small, $t_0+\rho^2<T$. If the ball meets a face, extend $w$ and $h$ oddly across it, using the same notation for the extensions. They satisfy the same heat equation weakly. The preceding barrier gives $|w|\leq C\delta^q$ throughout this cylinder, also after reflection.

Rescale by
\[
v(\mathbf y,s)=w(\mathbf x_0+\rho\mathbf y,t_0-\rho^2s),
\qquad Q_1=B_1(0)\times(-1,0).
\]
Then $v_s-\Delta_{\mathbf y}v=-\rho^2h(\mathbf x_0+\rho\mathbf y,t_0-\rho^2s)$. For $p>n+2$, the interior $W^{2,1}_p$ estimate \cite[Proposition 7.18]{lieberman1996second} gives, with $Q_{1/2}=B_{1/2}(0)\times(-1/4,0)$,
\[
\|v\|_{W^{2,1}_p(Q_{1/2})}
\leq C_p\bigl(\|v\|_{L^p(Q_1)}+\rho^2H\bigr)
\leq C\bigl(\delta^q+\rho^2H\bigr).
\]
Parabolic embedding \cite[Chapter II, Section 3, Lemma 3.3]{LadyzhenskayaSolonnikovUraltseva1968} and scaling back therefore yield
\[
\begin{aligned}
|\nabla_{\mathbf x}w(\mathbf x_0,t_0)|
&=\rho^{-1}|\nabla_{\mathbf y}v(0,0)|\\
&\leq C\rho^{-1}\|v\|_{W^{2,1}_p(Q_{1/2})}
\leq C\delta^{q-1}+CH\delta\longrightarrow0.
\end{aligned}
\]
Combining this with {(ii) and (iii)} and undoing the smooth transformations yields
\[
\nabla F^\epsilon\in C\bigl([1/M,M]^n\times(t_\epsilon,T]\bigr),
\qquad \nabla F^\epsilon(\mz,T)=\nabla u(\mz).
\]

\noindent\textbf{Step 2. Derivative bounds.}

\noindent\textbf{(i) Boundary derivative estimate.}
\begin{lemma}\label{lem boundary derivative estimate}
Let $F^\epsilon$ solve \eqref{equ PDE penal} with the prescribed Dirichlet and terminal data and the regularity established in Step 1. For $0<\epsilon\leq\min_i(\mu_i+\lambda_i)$ and sufficiently large $c_\beta$, independent of $\epsilon$,
\[
1-\mu_i-\epsilon\leq F_{z_i}^\epsilon(\mz,t)
\leq1+\lambda_i+\epsilon,
\qquad
\mz\in\partial[1/M,M]^n,\quad t_\epsilon<t\leq T,\quad 1\leq i\leq n.
\]
\end{lemma}
\begin{proof}

Our aim is to find sub- and supersolutions that bound $F^\epsilon$ and agree with it on the boundary face under consideration. Comparison and inward difference quotients then control the normal derivative of $F^\epsilon$ by the derivatives of these touching barriers.

\noindent\textbf{First, we construct a supersolution.}
The function $u(\mz)+\theta(T-t)$  is a viscosity supersolution of \eqref{equ PDE penal}, by the variational inequality for $u$. It agrees with the prescribed boundary and terminal data, so comparison gives
\[
F^\epsilon(\mz,t)\leq u(\mz)+\theta(T-t).
\]

\noindent\textbf{Second, we construct subsolutions.}
For each $i$, consider
\[
\begin{aligned}
&u(M,\mz_{-i})+\theta(T-t)
 +(1+\lambda_i+\epsilon)(z_i-M),\\
&u(1/M,\mz_{-i})+\theta(T-t)
 +(1-\mu_i-\epsilon)(z_i-1/M)
\end{aligned}
\]
The gradient constraints for $u$ place both functions below $u(\mz)+\theta(T-t)$ throughout the box, and hence below the prescribed data on the entire parabolic boundary. Lemma \ref{lem weighted hessian bound} bounds their non-penalty residuals $-F_t-\mathcal LF+\mathcal A(F)$ from above, independently of $\epsilon$ and $c_\beta$. Their upper and lower normal penalty terms, respectively, are at least $Mc_\beta$, since $\beta^\epsilon(\epsilon)=c_\beta$. Thus both are viscosity subsolutions for sufficiently large $c_\beta$, and comparison places them below $F^\epsilon$.

\noindent\textbf{Third, we derive the gradient bounds.}
The first subsolution touches $F^\epsilon$ at $z_i=M$, and the second at $z_i=1/M$. Taking inward difference quotients between these barriers and the supersolution, using the established boundary gradient continuity, gives
\[
\begin{aligned}
u_{z_i}(M,\mz_{-i})&\leq F_{z_i}^\epsilon(M,\mz_{-i},t)
             \leq1+\lambda_i+\epsilon,\\
1-\mu_i-\epsilon&\leq F_{z_i}^\epsilon(1/M,\mz_{-i},t)
              \leq u_{z_i}(1/M,\mz_{-i}).
\end{aligned}
\]
The gradient constraints for $u$ give the claimed normal derivative bounds. Tangential derivatives equal those of $u$ by the Dirichlet data. Gradient continuity extends the bounds to intersections of faces, and at $t=T$ they follow from $F^\epsilon=u$.

\end{proof}

\noindent\textbf{(ii) Interior derivative estimate.}
Set $G^{(k)}:=F_{z_k}^\epsilon$. Differentiating \eqref{equ PDE penal} in the interior gives
\begin{align}\label{equ G penal}
0={}&G_t^{(k)}+\frac12\sum_{i,j=1}^n\sigma_{ij}z_i z_j G_{z_i z_j}^{(k)}
  +\sum_{j=1}^n b_j^{(k)}G_{z_j}^{(k)}+(\alpha_k-r)G^{(k)}\notag\\
&-\gamma G^{(k)}\sum_{j=1}^n\sigma_{kj}z_jG^{(j)}
 +\beta^\epsilon(1-\mu_k-G^{(k)})-\beta^\epsilon(G^{(k)}-(1+\lambda_k)),
\end{align}
where
\[
\begin{aligned}
b_j^{(k)}={}&(\alpha_j-r+\sigma_{kj})z_j
 -\gamma z_j\sum_{i=1}^n\sigma_{ij}z_iG^{(i)}\\
&-(z_j+2M)(\beta^\epsilon)'(1-\mu_j-G^{(j)})
 +(2M-z_j)(\beta^\epsilon)'(G^{(j)}-(1+\lambda_j)).
\end{aligned}
\]
Step 1 and Lemma \ref{lem boundary derivative estimate} give
\[
G^{(k)}(\mz,T)=u_{z_k}(\mz),\qquad
1-\mu_k-\epsilon\leq G^{(k)}\leq1+\lambda_k+\epsilon
\quad\text{on }\partial[1/M,M]^n\times(t_\epsilon,T].
\]
In addition to the choice required in Lemma \ref{lem boundary derivative estimate}, take
\begin{align}\label{equ cbeta con}
c_\beta>\max_{1\leq k\leq n}\left\{
2(\alpha_k-r)(1+\lambda_k)
+4\gamma M(1+\lambda_k)\sum_{j=1}^n|\sigma_{kj}|(1+\lambda_j)
\right\}.
\end{align}
Define
\[
\mathcal E(t):=\max_{1\leq k\leq n}\sup_{\mz\in[1/M,M]^n}
\max\{G^{(k)}-(1+\lambda_k+\epsilon),\ (1-\mu_k-\epsilon)-G^{(k)},\ 0\}.
\]
If a violation occurs, choose a sufficiently small $\delta>0$ below its size and let
\begin{align}\label{equ com tdelta tep}
t_\delta:=\sup\{t\in(t_\epsilon,T):\mathcal E(t)\geq\delta\}.
\end{align}
Continuity and $\mathcal E(T)=0$ imply $t_\delta\in(t_\epsilon,T)$, $\mathcal E(t_\delta)=\delta$, and $\mathcal E(t)<\delta$ for $t>t_\delta$. Thus, throughout $[t_\delta,T]$,
\[
1-\mu_j-\epsilon-\delta\leq G^{(j)}\leq1+\lambda_j+\epsilon+\delta,
\qquad 1\leq j\leq n.
\]
Since $\epsilon\leq\mu_j+\lambda_j<1+\lambda_j$, this gives
$|G^{(j)}|\leq2(1+\lambda_j)+\delta$. By \eqref{equ cbeta con}, for sufficiently small $\delta$,
\[
\left|(\alpha_k-r)G^{(k)}-
\gamma G^{(k)}\sum_{j=1}^n\sigma_{kj}z_jG^{(j)}\right|<c_\beta
\]
for every $k$ on this time interval.

For some $k$ and $\tilde\mz\in[1/M,M]^n$,
\begin{align}\label{equ tilde G inter}
G^{(k)}(\tilde\mz,t_\delta)=1+\lambda_k+\epsilon+\delta
\quad\text{or}\quad
G^{(k)}(\tilde\mz,t_\delta)=1-\mu_k-\epsilon-\delta.
\end{align}
The boundary estimate in (i) forces $\tilde\mz$ to lie in the spatial interior. At an upper contact, $G_t^{(k)}\leq0$, $\nabla G^{(k)}=0$, and $D^2G^{(k)}\preceq0$. The lower penalty vanishes and
\[
\beta^\epsilon(G^{(k)}-(1+\lambda_k))
=\beta^\epsilon(\epsilon+\delta)\geq c_\beta.
\]
Hence the right-hand side of \eqref{equ G penal} is strictly negative, a contradiction. At a lower contact, $G_t^{(k)}\geq0$, $\nabla G^{(k)}=0$, and $D^2G^{(k)}\succeq0$, the upper penalty vanishes, and the lower penalty is at least $c_\beta$; the right-hand side is strictly positive, again a contradiction. Consequently,
\[
1-\mu_i-\epsilon\leq F_{z_i}^\epsilon\leq1+\lambda_i+\epsilon
\quad\text{on }[1/M,M]^n\times(t_\epsilon,T],\qquad 1\leq i\leq n.
\]
That is a contradiction for $t_\delta >-\infty$. In particular, \eqref{equ bound nabla Fep} holds globally.

\end{proof}

For solutions satisfying Lemma \ref{lem strong solu}, each penalty factor is bounded by $c_\beta$, and the gradient bounds control $\mathcal A(F^\epsilon)$. The barriers constructed in the proof of Lemma \ref{lem boundary derivative estimate} give an $\epsilon$-independent bound on $\|F^\epsilon\|_{L^\infty([1/M,M]^n\times(0,T))}$. Hence interior parabolic estimates give
\[
\|F^\epsilon\|_{W_p^{2,1}(U\times(t_0,t_1))}\leq C_{U,t_0,t_1,p},
\qquad U\Subset(1/M,M)^n,\quad 0<t_0<t_1<T,
\]
with constants independent of $\epsilon$.  A subsequence therefore satisfies
\[
F^\epsilon\rightharpoonup F^s
\quad\text{in }W^{2,1}_{p,\mathrm{loc}}((1/M,M)^n\times(0,T)).
\]
Taking $p>n+2$ also gives local uniform convergence of $F^\epsilon$ and its spatial gradients. Thus $\mathcal A(F^\epsilon)\to\mathcal A(F^s)$ locally uniformly.

\vspace{16pt}

{\bf Step 2.} Using the uniform estimates from Step 1, choose $p>n+2$ and extract a subsequence converging locally in $C^{1+\alpha,(1+\alpha)/2}$, $0<\alpha<1-(n+2)/p$, to a limit $F^s$ on time intervals bounded away from $T$. We compare this limit with $F^v$.

\noindent\textbf{(i) Viscosity solution in the interior.}
The nonnegative penalty terms give
\[
-F_t^\epsilon-\mathcal L F^\epsilon+\mathcal A(F^\epsilon)\geq0.
\]
The $\epsilon$-relaxed derivative bounds and local uniform convergence of the gradients yield the exact gradient constraints for $F^s$. Viscosity stability for the displayed inequality therefore implies that $F^s$ is a viscosity supersolution of \eqref{equ PDE F} on $(1/M,M)^n\times(0,T)$. We next verify the subsolution property there.

Let $\phi$ be a smooth test function such that $\phi-F^s$ has a local minimum zero at $(\hat\mz,\hat t)\in(1/M,M)^n\times(0,T)$. If $F_i^s(\hat\mz,\hat t)$ equals either $1-\mu_i$ or $1+\lambda_i$ for some $i$, then $\phi_i=F_i^s$ there, and
\[
\min\bigg\{-\phi_t-\mathcal L\phi+\mathcal A(\phi),
\min_{1\leq i\leq n}\{\phi_i-(1-\mu_i)\},
\min_{1\leq i\leq n}\{(1+\lambda_i)-\phi_i\}\bigg\}\leq0.
\]
Otherwise all gradient constraints are strict at this point. By the local uniform convergence of the spatial gradients, there are a neighborhood and $\zeta>0$ such that
\[
1-\mu_i+\zeta<F_i^\epsilon<1+\lambda_i-\zeta,\qquad 1\leq i\leq n,
\]
throughout that neighborhood for all sufficiently small $\epsilon$. When $\epsilon<\zeta$, every penalty vanishes, so
\[
-F_t^\epsilon-\mathcal L F^\epsilon+\mathcal A(F^\epsilon)=0.
\]

Viscosity stability yields the same equation for $F^s$ in this neighborhood and hence $-\phi_t-\mathcal L\phi+\mathcal A(\phi)\leq0$ at the contact point. Thus $F^s$ is a viscosity solution of \eqref{equ PDE F} on $(1/M,M)^n\times(0,T)$.

\noindent\textbf{(ii) Boundary and terminal conditions.}
Every $F^\epsilon$ has the prescribed Dirichlet values $F^v$ on the spatial boundary. The uniform spatial Lipschitz bounds for $F^\epsilon$ and $u$ therefore give
\[
|F^\epsilon(\mz,t)-F^v(\mz,t)|
\leq C\,\operatorname{dist}\bigl(\mz,\partial[1/M,M]^n\bigr),
\]
with $C$ independent of $\epsilon$. Passing to the locally uniform limit shows that $F^s$ extends continuously to the spatial boundary with the same values.

For the terminal data, Lemma \ref{lem weighted hessian bound} allows us to choose
\[
C\geq|\theta|,\qquad
C\geq\operatorname*{ess\,sup}_{\mz\in(1/M,M)^n}
|\mathcal Lu-\mathcal A(u)|.
\]
The functions $u+C(T-t)$ and $u-C(T-t)$ have zero penalty terms and are, respectively, a supersolution and a subsolution of \eqref{equ PDE penal}, bracketing all prescribed data. Comparison yields
\[
|F^\epsilon(\mz,t)-u(\mz)|\leq C(T-t),
\]
uniformly in $\epsilon$, so $F^s(\mz,T)=u(\mz)$ by continuous extension. Thus $F^s$ and $F^v$ have the same boundary and terminal values.

\noindent\textbf{(iii) Comparison: $F^s=F^v$.}
Applying the strict comparison argument of \cite[Sections 5.C and 8]{CrandallIshiiLions1992}, with an affine perturbation making the gradient constraints strict, gives $F^s=F^v$. The estimates from Step 1 therefore apply to $F^v=u+\theta(T-t)$ and yield the asserted local Sobolev regularity of $F^v$ and $u$.

\vspace{10pt}

{\bf Step 3.} The $n$-dimensional Hessian estimate.
By Assumption \ref{ass positive no trading}, $\mathcal L$ is uniformly elliptic on $\mathbf N$; together with the local $C^{1,1}$ regularity from Lemma \ref{lem weighted hessian bound}, the standard interior bootstrap argument gives $u\in C^\infty(\mathbf N)$. In this region,
\[
\theta-\mathcal Lu+\mathcal A(u)=0.
\]
Lemma \ref{lem weighted hessian bound} and $z_i\geq1/M$ on $\mathbf N$ give $|u_{z_i z_j}|\leq M^2C$ for all $i,j$, and hence the uniform Hessian bound.

\end{proof}

\section{Theoretical Analysis: Preliminary results for Further analysis}\label{sec OP and FB}

In this section, we shall first consider non-corner points on the free boundary and show that the problem is locally equivalent to an obstacle problem. After that, we will show that the free boundaries are Lipschitz.

\subsection{Regularity at non-corner points}

In this subsection, we consider a point $P: = (P_i)_{1\leq i \leq n}$ with $P_1 = l^+_1(P_{-1})$. Moreover, we shall require $u_{z_1} = 1-\mu_1$, and $u_{z_i} \in (1-\mu_i, 1+\lambda_i)$, $\forall i \neq 1$ at this point $P$ with $z_i \neq 0$, $\forall i$.

By Lemma \ref{lem weighted hessian bound}, the full gradient $\nabla u$ is continuous near $P$. Hence, for sufficiently small $\epsilon_1,\epsilon>0$, the rectangle $R:=P+(-\epsilon_1,\epsilon_1)\times(-\epsilon,\epsilon)^{n-1}$ satisfies $\overline R\subset(\mathbb R\setminus\{0\})^n$, $u_{z_1}<1+\lambda_1$, and $1-\mu_i<u_{z_i}<1+\lambda_i$ for every $i\ne1$ throughout $R$.

\begin{proposition}\label{prop transf}
There are $\epsilon_1, \epsilon>0$, such that the function
$u \in W^{2}_p(R)$ satisfies
\begin{align}
\min \bigg\{\theta- \mathcal{L} u + \mathcal{A}(u),
\quad  u_{z_1}- (1-\mu_1)  \bigg\}= 0,\ \forall {\bf z} \in   R.
\end{align}
We have $u_{z_1}> 1-\mu_1$ on the boundary $z_1 = P_1-\epsilon_1$.
\end{proposition}
\begin{proof}[\bf Proof of Proposition \ref{prop transf}]
We only need to make sure $u_{z_1}> 1-\mu_1$ on the boundary. By definition, at $(P_1-\epsilon_1, P_2, ..., P_n)$, we have $u_{z_1} > 1-\mu_1$, while $u_{z_1}$ is continuous, and we can make $\epsilon$ small enough to make sure that on the left boundary $u_{z_1}> 1-\mu_{1}$.
\end{proof}

We denote $z_1^L = P_1 -\epsilon_1$, $z_1^H = P_1 +\epsilon_1$, and  $z_i^L = P_i -\epsilon$, $z_i^H = P_i +\epsilon$, $i\neq 1$. Then the region under consideration is
$R= \prod\limits_i (z_i^L, z_i^H)$.

Define $G = u_{z_1}$. By direct differentiation, we see that, on the one hand
\begin{align}\label{equ G ope NN}
    \mathcal{L}_2 G + B = 0 \text{ and } G > 1-\mu_1, \text{ in the no-trading region $\mathbf N \cap R$.}
\end{align}
where
$$\mathcal{L}_2 G = \frac{1}{2} \sum_{i, j} \sigma_{ij} z_i z_j G_{z_iz_j} + \sum_{i}\sigma_{i1} z_i G_{z_i}
$$
is a linear operator and the function
\begin{align}
    B({\bf z}) = & (\alpha_1-r)G +\sum_{i} (\alpha_i -r) z_i G_{z_i} \\
    & - \gamma \Big(\sum_{i, j} \sigma_{ij} z_i z_j u_{z_i} G_{z_j} +\sum_i \sigma_{1i} z_i u_{z_i} G\Big)
\end{align}
is bounded in $\mathbf N \cap R$.

On the other hand, we have
\begin{align}
    G  = 1-\mu_1 \text{ in ${\bf S_1} \cap R$.}
\end{align}
Moreover, since $G $ is constant in $\mathbf S_1 \cap R$, we see that $\mathcal{L}_2 G = 0$. Thus, we expect that $G$ satisfies the following variational inequality.
\begin{align}
\begin{cases}\label{equ G opera}
    \min \{-\mathcal{L}_2 \tilde{G} - \tilde B,\quad  \tilde {G}-(1-\mu_1)\} = 0, \quad
    &\forall \mz \in R\\
     \tilde {G} = u_{z_1}, \quad & \forall \mz \in \partial R,
\end{cases}
\end{align}
where
\begin{align}
\tilde{B}(\mz) =
\begin{cases}
{B}(\mz), & \text{ if $z_1 < l_1^+(z_{-1})$}, \\
0, & \text{ if $z_1 \geq l_1^+(z_{-1})$}.
\end{cases}
\end{align}
\begin{lemma}\label{lem u1 strong}
The problem \eqref{equ G opera} permits a unique solution in $W^{2}_{\infty, loc}(R)$, which is $u_{z_1}$.
\end{lemma}
\begin{proof}[\bf Proof of Lemma \ref{lem u1 strong}]

{\bf Step 1}: We first establish the existence of a strong solution $\hat G$ to \eqref{equ G opera}.
Since $\overline R\subset(\mathbb R\setminus\{0\})^n$, the operator $\mathcal L_2$ has smooth coefficients and is uniformly elliptic on $R$. By Lemma \ref{lem weighted hessian bound}, $\tilde B\in L^\infty(R)$ and $G\in C^{0,1}(\overline R)$, with $G\geq1-\mu_1$. The standard elliptic obstacle penalization method (see \cite[Chapter 3]{BensoussanLions1982}) therefore gives a solution
\[
\hat G\in C(\overline R)\cap W^{2,p}_{\mathrm{loc}}(R),\qquad 1<p<\infty,
\]
to \eqref{equ G opera}. In this construction, comparison bounds the penalty terms uniformly, and the interior elliptic estimates in \cite[Chapter 9]{gilbarg1998elliptic} give the local $W^{2,p}$ bounds. Boundary barriers for the rectangle ensure that the limit attains $\hat G=G$ continuously on $\partial R$.

\vspace{16pt}
{\bf Step 2}. We next show that $W^2_{p}(\mathcal K)< \infty$, for any compact subset $\mathcal K$ of $R\cup\bigl(\{z_1^L\}\times\prod_{i=2}^n(z_i^L,z_i^H)\bigr)$ and $1<p< \infty$. That is, the compact set can intersect the open left face of the rectangle.

The proof is that the open left face $\{z_1^L\}\times\prod_{i=2}^n(z_i^L,z_i^H)$ is in the $\bf NN$ region, on which the value function is smooth.  Then, by the classical boundary estimate (see \cite{han2011elliptic}), we have the result.

\vspace{16pt}
{\bf Step 3}: We prove $\hat G \geq G$ in $R$.
In $\mathbf S_1\cap R$, the obstacle constraint gives $\hat G\geq1-\mu_1=G$. Moreover, \eqref{equ G opera} gives
\[
-\mathcal L_2\hat G-\tilde B\geq0\qquad\text{a.e. in }R.
\]
Combining this with \eqref{equ G ope NN}, we obtain $\mathcal L_2(G-\hat G)\geq0$ in $\mathbf N\cap R$. On the boundary of this region, $G-\hat G\leq0$: on the free-boundary part, $G=1-\mu_1\leq\hat G$, while on $\partial R$ the prescribed values agree. The weak maximum principle (see \cite[Theorem 9.1]{gilbarg1998elliptic}), applied to the positive level sets of $G-\hat G$, yields $G\leq\hat G$ in $\mathbf N\cap R$. Thus $\hat G\geq G$ throughout $R$.

\vspace{16pt}

{\bf Step 4}: We show that $\hat G = G$ in $R$.

Define
\begin{align}
\hat{u}(z_1, \cdot) = u(z^L_1, \cdot) + \int_{z^L_1}^{z_1} \hat{G}(v, \cdot) dv,\ \forall \ (z_1, \cdot) \in R
\end{align}
and
\begin{align}
f(\mz) =
    \begin{cases}
    \sum_{i} (\alpha_i-r) z_i u_{z_i} - \mathcal{A}(u)& \text{ if $z_1 \leq l_1^+(z_{-1})$} \\
    f(l_1^+(z_{-1}), z_{-1}), & \text{ if $z_1 > l_1^+(z_{-1})$}. \label{equ f defi}
    \end{cases}
\end{align}
We have the following lemma.
\begin{lemma}\label{lem hatu sub}
  The function  $\hat u\in W^{2}_{p, loc}(R)$, $1<p< \infty$, satisfies
\begin{align}\label{equ pde hat u}
 \theta -\frac{1}{2} \sum_{i, j} \sigma_{ij} z_i z_j \hat{u}_{z_iz_j} - f \leq 0, \text{ a.e. in } R
\end{align}
with the boundary conditions
\begin{align}\label{equ bd hat u}
\begin{cases}
\hat u = u, & \text{on } \partial R\setminus \{ z_1 = z^H_1\}, \\
\hat u_{z_1} = u_{z_1},  & \text{on } \{ z_1 = z^H_1\}.
\end{cases}
\end{align}
\end{lemma}
\begin{proof}[\bf Proof of Lemma \ref{lem hatu sub}]
(1). The boundary conditions can be directly derived from the definition.

(2). We verify the regularity below.

We easily see that $\hat u_{z_1 z_i} = \hat G_{z_i}$, $1\leq i\leq n$, which is continuous.
For $i, j \neq 1$, we see that
$\hat u_{z_i z_j}(z_1, \cdot) = u_{z_i z_j}(z_1^L,z_{-1})+ \int_{z_1^L}^{z_1} \hat G_{z_i z_j}(s, \cdot) ds$.

Set $g: = \hat G_{z_i z_j}$. Consider any compact rectangle $[z_1^L, \tilde{z}] \times R_{-1}$, where $R_{-1} \subset \prod_{2\leq i \leq n} (z_i^L, z_i^H)$. Since $u_{z_i z_j}(z_1^L,z_{-1})$ is uniformly bounded, we have, from the triangle inequality and Minkowski's integral inequality, that
\begin{align}
    & \Big(\int_{[z_1^L, \tilde{z}] \times R_{-1}} |\hat u|^p_{z_i z_j} (\mz) d \mz\Big)^{1/p}\\
 \leq  & \Big(\int_{[z_1^L,\tilde{z}] \times R_{-1}} \big(C+ \int_{z_1^L}^{z_1} |g|(s, z_{-1}) ds   \big)^p d \mz\Big)^{1/p} \\
    \leq  &  \Big(\int_{[z_1^L,\tilde{z}] \times R_{-1}} (C+ \int_{z_1^L}^{\tilde z} |g|(s, z_{-1}) \mathbf{1}_{s \leq z_1} ds   )^p d \mz\Big)^{1/p}\\
 \leq & C+ \Big(\int_{[z_1^L,\tilde{z}] \times R_{-1}} ( \int_{z_1^L}^{\tilde z} |g|(s, z_{-1}) \mathbf{1}_{s \leq z_1} ds   )^p d \mz\Big)^{1/p}\\
    \leq & C+ \int_{z_1^L}^{\tilde z} \Big( \int_{[z_1^L, \tilde z] \times R_{-1}}  |g|^p(s, z_{-1}) \mathbf{1}_{s \leq z_1} d \mz  \Big)^{1/p} ds\\
 \leq & C+ \int_{z_1^L}^{\tilde z} \Big( \int_{[z_1^L, \tilde z] \times R_{-1}}  |g|^p(s, z_{-1})d \mz  \Big)^{1/p} ds\\
    \leq & C+ \int_{z_1^L}^{\tilde z} \Big( (\tilde z-z_1^L)\int_{ R_{-1}}  |g|^p(s, z_{-1})d \mz  \Big)^{1/p} ds\\
 =  & C+  (\tilde z-z_1^L)^{1/p}  \int_{z_1^L}^{\tilde z} \Big(\int_{ R_{-1}}  |g|^p(s, z_{-1})d \mz  \Big)^{1/p} ds\\
 \leq & C+  C  \Big(\int_{z_1^L}^{\tilde z} \int_{ R_{-1}}  |g|^p(s, z_{-1})d \mz  ds\Big)^{1/p}\\
 = & C(1+ \|\hat G_{z_i z_j}\|_{[z_1^L, \tilde{z}] \times R_{-1} })\\
 < & \infty.
\end{align}

(3). We will prove \eqref{equ pde hat u} in what follows.

Define the sets $\hat {\bf NN} : = \{\mz \in R| \hat G > 1-\mu_1\}$ and $\hat {\bf S}_1 : = \{\mz \in R| \hat G = 1-\mu_1\}$.
On the one hand, when in $\hat {\bf NN}$, we see from the variational inequality that $\mathcal{L}_2 \hat G + \tilde B = 0$.
On the other hand, when in $\hat {\bf S}_1$, from $1-\mu_1= \hat G \geq G \geq 1-\mu_1$ we have $G = 1-\mu_1$, and thus
$\tilde B = 0$. Moreover, since for $p>n$, the classical second-order derivatives exist a.e. (see Theorem 6.5 of \cite{evans1991measure}), and from the Lebesgue differentiation theorem, almost every point in $\hat {\bf S}_1$ is a Lebesgue point. Therefore,  we have
$\mathcal {L}_2 \hat G = 0$ a.e.\text{ in $\hat {\bf S}_1$}.
In summary, we have
\begin{align}\label{equ a.e. operator hatG}
    \mathcal{L}_2 \hat G + \tilde B = 0, a.e. \text{ in } R.
\end{align}

We immediately see from $\hat G\geq G$ in step 3 that
\begin{align}\label{equ hat F>F}
\hat  u \geq u\text{ in }R.
\end{align}
and therefore
\begin{align}\label{equ Fyy yl}
 \hat u_{z_1z_1} = \hat{G}_{z_1} \geq {G}_{z_1} = u_{z_1z_1} \text{on the open left face}.
\end{align}

On the open left face $\{z_1^L\}\times\prod_{i=2}^n(z_i^L,z_i^H)$, we also have
\begin{align}
&\hat{u}_{z_1z_i} = \hat {G}_{z_i} = G_{z_i}= {u}_{z_1z_i},\quad \forall i \neq 1, \\
&\hat{u}_{z_i z_j}={u}_{z_i z_j},\quad \forall i,j\neq1,\notag\\
& \hat{u}_{z_i} = {u}_{z_i}, \forall 1\leq i\leq n. \label{equ deriva left}
\end{align}
The following integration uses the fundamental theorem for Sobolev functions along almost every $z_1$-line: $f_{z_1}=\tilde B$ by \eqref{equ f defi}, and the weak $z_1$-derivative of the second-order term equals $\mathcal L_2\hat G$.
Equations \eqref{equ a.e. operator hatG} and \eqref{equ Fyy yl}-\eqref{equ deriva left} give
\begin{align}
&- \theta + \frac{1}{2} \sum_{1\leq i, j\leq n} \sigma_{ij} z_i z_j \hat u_{z_i z_j} + f \Big|_{(z_1, \cdot)} \\
    = & - \theta + \frac{1}{2} \sum_{1\leq i, j\leq n} \sigma_{ij} z_i z_j \hat u_{z_i z_j} + f  \Big|_{(z^L_1, \cdot)} \\
    &\qquad + \int_{z_1^L}^{z_1} \partial_1\bigg( \frac{1}{2}\sum_{1\leq i, j\leq n} \sigma_{ij} z_i z_j \hat u_{z_i z_j} +f\bigg) (v, \cdot)  dv             \\
= & - \theta + \frac{1}{2} \sum_{1\leq i, j\leq n} \sigma_{ij} z_i z_j \hat u_{z_i z_j} + f  \Big|_{(z^L_1, \cdot)} + \int_{z_1^L}^{z_1} ( \mathcal{L}_2 \hat G+ \tilde B)(v, \cdot)  dv \\
    = &  - \theta + \frac{1}{2} \sum_{1\leq i, j\leq n} \sigma_{ij} z_i z_j \hat u_{z_i z_j} + f  \Big|_{(z^L_1, \cdot)}  \\
\geq &- \theta + \frac{1}{2} \sum_{1\leq i, j\leq n} \sigma_{ij} z_i z_j u_{z_i z_j} + f  \Big|_{(z^L_1, \cdot)}  \\
    = & 0 \quad a.e. \text{ in $R$}.
\end{align}
\end{proof}

\begin{lemma}\label{lem u sup}
  The function  $u\in W^{2}_p(R)$, $1<p< \infty$, satisfies
\begin{align}\label{equ pde u modi}
 \theta - \frac{1}{2} \sum_{1\leq i, j\leq n} \sigma_{ij} z_i z_j u_{z_i z_j} - f \geq 0, \text{ a.e. in } R
\end{align}
with the boundary conditions
\eqref{equ bd hat u}.
\end{lemma}

\begin{proof}
From the variational inequality we have $u\in W^{2}_p(R)$ satisfies
\begin{align}
 \theta - \frac{1}{2} \sum_{1\leq i, j\leq n} \sigma_{ij} z_i z_j u_{z_i z_j} - f = 0, \text{ if } z_1 < l_1^+(z_{-1}).
\end{align}
We shall focus on the case $z_1 > l_1^+(z_{-1})$ in what follows.
From Lebesgue's differentiation theorem, we have for almost every $z_{-1} \in\prod_{i = 2}^n (P_i-\epsilon, P_i+\epsilon)$,
\begin{align}\label{equ LDT fb}
    \lim \limits_{\delta \to 0} \frac{1}{|B_\delta(z_{-1})|} \int_{B_\delta(z_{-1})} l_1^+(z) dz = l_1^+(z_{-1}),
\end{align}
where $B_\delta(z_{-1})$ is the ball centered at $z_{-1}$ with radius $\delta$.
Combined with Corollary \ref{coro semiconti fb}, we have the following proposition.
\begin{proposition}\label{prop cone sr}
    For almost every point $\mz^0 \in R$ such that $l_1^+(z^0_{-1})< z^0_1$, any $\epsilon>0$, and any unit vector ${\bf n}\in \mathbb{R}^n$, there are points $\{\mz^i\}_{i\geq 1} \subset \mathbf{S}_1\cap R$ such that $\mz^i \to \mz^0$ and $\frac{\mz^i - \mz^0}{\|\mz^i - \mz^0\|} \cdot \bf n \geq 1-\epsilon$, $\forall i\geq 1$.
\end{proposition}
\begin{proof}[Proof of Proposition \ref{prop cone sr}]
    We focus on those points $\mz$ such that \eqref{equ LDT fb} holds. We prove by contradiction. Assume that, for some point $\mz^0$, the set $\mathcal{C}\cap (\mathbf{S}_1\cap R) =\emptyset$, where  $$\mathcal{C}: = \Big\{\mz\ \Big| \|\mz -\mz^0\| <r, \frac{\mz - \mz^0}{\|\mz - \mz^0\|} \cdot \bf n \geq 1-\epsilon  \Big\} , $$
    and $r< \frac{z^0_1- l_1^+(z^0_{-1})}{2}$.
    This implies that the free boundaries $l_1^+(z_{-1}) \geq z^0_1-r > l_1^+(z^0_{-1})+r$ for any $\mz \in \mathcal{C}$. Noticing that the measure of the projection of the cone on the plane $\{\mz | \mz_1 = 0\}$ is proportional to $r^{n-1}$, this yields a contradiction to  \eqref{equ LDT fb} using Corollary \ref{coro semiconti fb}.
\end{proof}

Based on the above proposition, we next have the following proposition.
\begin{proposition}\label{prop D2u fb}
    For almost every $\mz \in R\cap {\bf S_1}$, there is a sequence $w^k \downarrow l^+_1(z_{-1})$, $k\geq 1$,  such that
\begin{align}
    \nabla u(\mz ) =  \nabla u(w^k, z_{-1}),\quad   \nabla^2 u(\mz ) =  \nabla^2 u(w^k, z_{-1})
\end{align}
\end{proposition}
\begin{proof}[Proof of Proposition \ref{prop D2u fb}]
    By Fubini's theorem, for almost every $\mz\in R\cap\mathbf S_1$ at which $u$ is second-order differentiable, we can choose $w^k\downarrow l_1^+(z_{-1})$, with $w^k<z_1$, such that $u$ is second-order differentiable, Proposition \ref{prop cone sr} applies, and the variational inequality \eqref{equ PDE u} holds at $(w^k,z_{-1})$. We next show that $u_{z_iz_j}(\mz) =  u_{z_iz_j}(w^k, z_{-1})$.

    Fix $w = w^k< z_1$. We argue by contradiction.
    We first prove $\nabla u(\mz) = \nabla u(w^k, z_{-1})$.

    Assume $\max\limits_{\|\bf n\| = 1}  \Big(\nabla u(\mz) -  \nabla u(w, z_{-1}) \Big) \cdot \bf n > 0$, then there is a unit vector $\bf n$ and $\epsilon, \delta>0$, such that
    \begin{align}\label{equ nabla u posi}
      \Big(\nabla u(\mz) -  \nabla u(w, z_{-1}) \Big) \cdot \mathbf n' > \delta >0  \text{ for any unit vector ${\bf n}'$ such that $\bf n' \cdot n \geq 1-\epsilon$.  }
    \end{align}

    Choose a sequence $\mz^i$ as in Proposition \ref{prop cone sr} which converges to $(w, z_{-1})$, then we have from $u_{z_1} = 1-\mu_1$ that
    $$ u(\mz^i  + (z_1 - w) e_1 ) =  u(\mz^i) + (1-\mu_1)  (z_1 - w)  $$
    and
    $$ u(\mz) =  u(w, z_{-1}) + (1-\mu_1)  (z_1 - w)  $$
    Taking the difference, we have
    \begin{align}
        & u(\mz^i  + (z_1 - w) e_1 ) -  u(\mz) =  u(\mz^i) - u(w, z_{-1}) \label{equ taylor term} \\
        & \nabla u(\mz) \cdot l^i  + o( \|l^i\|) =  \nabla u(w, z_{-1}) \cdot l^i  + o( \|l^i\|)
    \end{align}
    where $l^i = (\mz^i - (w, z_{-1}))$. That contradicts \eqref{equ nabla u posi}.

    After deriving the equation for the first-order derivative, we can similarly prove the second-order derivative equation by second-order expansion of equation \eqref{equ taylor term}.

\end{proof}

Therefore, from Proposition \ref{prop D2u fb},
we have for almost every $\mz \in R\cap \bf S_1 $,
\begin{align}
 &  \theta- \frac{1}{2} \sum_{1\leq i, j\leq n} \sigma_{ij} z_i z_j u_{z_i z_j}|_{(z_1, z_{-1})} \\
 = & \lim \limits_{i \to \infty } \Big(\theta-  \frac{1}{2}\sum_{1\leq i, j\leq n} \sigma_{ij} z_i z_j u_{z_i z_j}\Big)|_{(w_i, z_{-1})}\\
 \geq & \lim \limits_{i\to \infty} \Big(\sum_{i = 1}^n (\alpha_i-r) z_i u_{z_i}  - \mathcal{A} (u)\Big)|_{(w_i, z_{-1})} \label{equ limit vari u}\\
 = & f(l_1^+(z_{-1}), z_{-1}) \label{equ limit f}\\
 = & f(\mz),\label{equ limit f 2}
\end{align}
where the inequality \eqref{equ limit vari u} is from the variational inequality \eqref{equ PDE u} , \eqref{equ limit f} is from the continuity of $\nabla u$, and the equation \eqref{equ limit f 2} is from the definition \eqref{equ f defi} of function $f$.

In summary, the function $u\in W^{2}_p(R)$ satisfies
\begin{align}
\theta - \frac{1}{2} \sum_{1\leq i, j\leq n} \sigma_{ij} z_i z_j u_{z_i z_j} - f\geq 0, \text{ a.e. in $R$}.
\end{align}
\end{proof}
From Lemmas \ref{lem hatu sub} and \ref{lem u sup}, we have from the mixed Dirichlet--Neumann comparison principle (see \cite[Problem 3.1(a) and Theorem 9.1]{gilbarg1998elliptic}) that $u\geq \hat {u}$. Noticing that by definition of $\hat {u}$ and $\hat G \geq G$ derived in Step 3, we have $\hat u\geq u$, then we shall see $u = \hat u$, which also implies $\hat G = \hat u_{z_1} = u_{z_1} = G$.

\vspace{16pt}
{\bf Step 5}. We finally show that $G \in W^2_{\infty, loc}(R)$.

Choose $p>n$. The established $W^{2,p}_{\mathrm{loc}}$ regularity gives $G\in C^{1,\alpha}_{\mathrm{loc}}(R)$ for some $\alpha>0$, and Lemma \ref{lem weighted hessian bound} then gives $B\in C^\alpha_{\mathrm{loc}}(R)$. By Step 4 and \eqref{equ a.e. operator hatG}, $\mathcal L_2G=-\tilde B\in L^\infty(R)$. Since $G-(1-\mu_1)\geq0$, the Harnack inequality yields, at each contact point $\mz^0\in\{G=1-\mu_1\}$ and for sufficiently small $r$,
\[
\sup_{B_r(\mz^0)}\bigl(G-(1-\mu_1)\bigr)\leq Cr^2,
\]
with locally uniform $C$. Then the Schauder estimate implies the locally uniform bound of $D^2G$ (see \cite[Lemma 1 and the proof of Theorem 1]{ros2020regularity}), and hence $G\in W^{2,\infty}_{\mathrm{loc}}(R)$.
\end{proof}

\begin{figure}
\begin{minipage}{\linewidth}\centering		\includegraphics[width=0.8\textwidth]{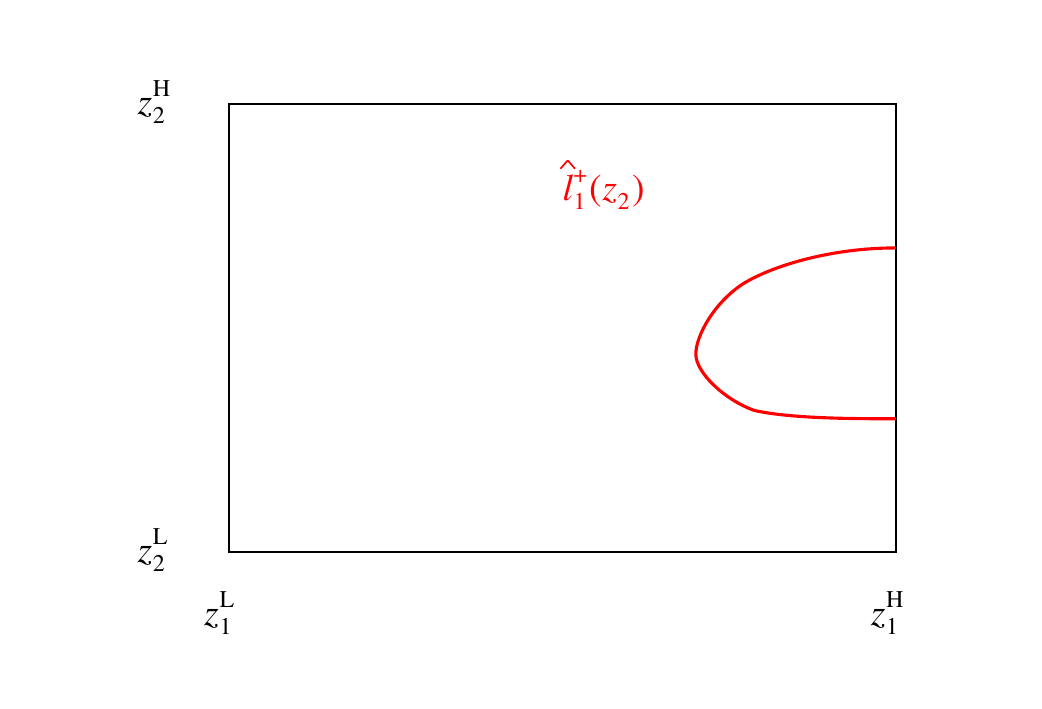}
\end{minipage}
\caption{Illustrative figure for $\hat l_1^+(z_2)$}
\end{figure}

Similar arguments also work for non-corner points on other free boundaries. In summary, Lemma \ref{lem u1 strong} can be extended to the following proposition.
\begin{proposition}\label{Prop F_i regu}
For each $1\leq i\leq n$, define
\[
\Omega_i:=\left\{\mz\in(1/M,M)^n:\
1-\mu_j<u_{z_j}(\mz)<1+\lambda_j\ \text{for every }j\ne i\right\}.
\]
Then $u_{z_i}\in W^{2,\infty}_{\mathrm{loc}}(\Omega_i)$.
\end{proposition}

\subsection{Regularity across free boundary}\label{sec Lipschitz}
In this subsection, we will consider the free boundaries $l^\pm_i$, $1\leq i\leq n$ and prove them to be Lipschitz at non-corner points. Based on this Lipschitz regularity, we prove that $u$ is locally $C^{2,\alpha}$, but cannot be $C^3$, at non-corner points.

\begin{theorem}\label{thm local lipschitz selling}
Let $\mz^0\in(1/M,M)^n$ be any non-corner point on a free boundary $z_i=l_i^\pm(z_{-i})$, where $1\leq i\leq n$. Then $l_i^\pm$ is Lipschitz in a neighborhood of $z_{-i}^0$. Moreover, the function $u$ is locally $C^{2, \alpha}$, $\forall 0< \alpha <1$,  in this neighborhood.
\end{theorem}
\begin{proof}
Without loss of generality, consider $\mz^0=(l_1^+(z_{-1}^0),z_{-1}^0)$ on the selling boundary of asset $1$, and choose the neighborhood $R$ as above. Set $G=u_{z_1}$. We first prove that there exist $\rho>0$ and $C\geq1$, with $B_\rho(\mz^0)\Subset R$, such that
\begin{equation}\label{equ local derivative comparison}
|\nabla_{z_{-1}}G|\leq C(-G_{z_1})
\qquad\text{in }B_\rho(\mz^0).
\end{equation}

\noindent\textbf{(i) A strictly signed equation for $G_{z_1}$.}
By Lemmas \ref{lem weighted hessian bound} and \ref{lem u1 strong}, $u,G\in C^{1,1}_{\mathrm{loc}}(R)$, and both are smooth in $\mathbf N\cap R$. Set $H=G_{z_1}$. Concavity gives $H\leq0$. Since $G\geq1-\mu_1$, we also have $\nabla G=0$ on the contact set $\{G=1-\mu_1\}$.
Differentiating the equation for $u$ twice in $z_1$ gives, in $\mathbf N\cap R$,
\begin{equation}\label{equ normal derivative equation}
\mathscr L_0H=\gamma\left(
\sum_{i,j}\sigma_{ij}z_i z_jG_{z_i}G_{z_j}
+2G\sum_i\sigma_{1i}z_iG_{z_i}+\sigma_{11}G^2\right),
\end{equation}
where
\begin{align*}
\mathscr L_0\phi={}&\frac12\sum_{i,j}\sigma_{ij}z_i z_j\phi_{z_i z_j}
+\sum_i z_i\left(\alpha_i-r+2\sigma_{1i}
-\gamma\sum_j\sigma_{ij}z_j u_{z_j}\right)\phi_{z_i}\\
&+\left(\sigma_{11}+2(\alpha_1-r)
-2\gamma\sum_i\sigma_{i1}z_i u_{z_i}\right)\phi.
\end{align*}
At $\mz^0$, the right-hand side of \eqref{equ normal derivative equation} equals $\gamma\sigma_{11}(1-\mu_1)^2>0$. By continuity, choose $B_{2\rho}(\mz^0)\Subset R$ so that it is bounded below by $c_*>0$ there. Choose $C_0\geq0$ no smaller than the supremum of the zeroth-order coefficient of $\mathscr L_0$ on this ball, and set $\mathscr L=\mathscr L_0-C_0$. This operator has nonpositive zeroth-order coefficient, and
\begin{equation}\label{equ strict normal derivative source}
\mathscr LH=\mathscr L_0H-C_0H\geq c_*>0
\qquad\text{in }\mathbf N\cap B_{2\rho}(\mz^0).
\end{equation}

\noindent\textbf{(ii)  Lipschitz continuity of free boundaries.}
The formula for $B$ and the preceding $C^{1,1}$ bounds show that $B$ is Lipschitz on the ball. For $j\ne1$, differentiating \eqref{equ G ope NN} gives
\[
\mathcal L_2G_{z_j}
=-B_{z_j}-\sum_i\sigma_{ij}z_iG_{z_i z_j}
-\sigma_{j1}G_{z_j}
\qquad\text{in }\mathbf N\cap B_{2\rho}(\mz^0).
\]
Its right-hand side is bounded. Since $\mathscr L-\mathcal L_2$ contains only first- and zeroth-order terms with bounded coefficients, $\mathscr LG_{z_j}$ is also bounded. Choose $0\leq\chi\leq1$ in $C_c^\infty(B_{2\rho}(\mz^0))$ with $\chi=1$ on $B_\rho(\mz^0)$. The product rule and the bound on $D^2G$ give a constant $C_1$ such that
\[
|\mathscr L(\chi G_{z_j})|\leq C_1,\qquad j\ne1.
\]
For $K>C_1/c_*$, \eqref{equ strict normal derivative source} yields
\[
\mathscr L\bigl(-KH\pm\chi G_{z_j}\bigr)
\leq-Kc_*+C_1<0
\quad\text{in }\mathbf N\cap B_{2\rho}(\mz^0).
\]
These functions are continuous up to the boundary and nonnegative there: $\nabla G=0$ on the contact boundary, while $\chi=0$ on the outer sphere and $-H\geq0$. A negative minimum would occur in the interior, where the nonpositive zeroth-order coefficient would imply $\mathscr L(-KH\pm\chi G_{z_j})\geq0$, a contradiction. Hence
\[
|G_{z_j}|\leq K(-G_{z_1})\qquad\text{in }\mathbf N\cap B_\rho(\mz^0).
\]
The same inequality holds on the contact set because $\nabla G=0$. Combining the component bounds proves \eqref{equ local derivative comparison}, after increasing $C$ if necessary. Therefore, the free boundary is locally Lipschitz.

(iii). The $C^{2, \alpha}$ regularity is straightforward from the classical Schauder estimate.

\end{proof}

\begin{remark}\label{rem free boundary higher regularity}
For the classical obstacle problem with the Laplacian and smooth data, the regular part of the free boundary is $C^\infty$; see \cite{Caffarelli1977Regularity}. In the present setting, the boundary Harnack results of \cite[Corollary~1.7]{ros2021new} yield local $C^{1,\alpha}$ regularity of the free boundary near regular non-corner points, for some $\alpha\in(0,1)$.
\end{remark}

\begin{theorem}\label{thm failure of C3 regularity}
At every non-corner free-boundary point $\mz^0\in(1/M,M)^n$, the value function $u$ is not $C^3$ in any neighborhood of $\mz^0$.
\end{theorem}

\begin{proof}
Without loss of generality, consider a non-corner point $\mz^0=(l_1^+(z_{-1}^0),z_{-1}^0)$ on the selling boundary $l_1^+$, and choose the neighborhood $R$ as above. Set $G=u_{z_1}$. We use the following lemma.

\begin{lemma}\label{lem B nega}
    If $\mz^0$ is a free boundary point in $R$, then $B(\mz^0) <0$.
\end{lemma}

By Lemma~\ref{lem B nega}, $B(\mz^0)<0$. Suppose, for a contradiction, that $u$ is $C^3$ in a neighborhood of $\mz^0$. Then $G$ is $C^2$ there. Since $G=1-\mu_1$ in $\mathbf S_1\cap R$, continuity implies $\nabla G(\mz^0)=\nabla^2G(\mz^0)=0$, and hence $\mathcal L_2G(\mz^0)=0$. On the other hand, passing to the boundary from $\mathbf N\cap R$ in \eqref{equ G ope NN} gives
\[
\mathcal L_2G(\mz^0)=-B(\mz^0)>0,
\]
a contradiction. This proves the theorem.
\end{proof}

\begin{proof}[Proof of Lemma~\ref{lem B nega}]
First, $B(\mz^0)\leq0$. If $B(\mz^0)>0$, continuity gives a ball $B_r(\mz^0)\Subset R$ on which $B>0$. By \eqref{equ a.e. operator hatG} and $\hat G=G$,
\[
\mathcal L_2\bigl(G-(1-\mu_1)\bigr)
=-\tilde B=-B\mathbf1_{\mathbf N}\leq0
\qquad\text{a.e. in }B_r(\mz^0).
\]
Since $G-(1-\mu_1)\geq0$ and vanishes at $\mz^0$, the strong minimum principle (see \cite{gilbarg1998elliptic}) implies $G\equiv1-\mu_1$ in this ball, contradicting that $\mz^0$ is a free boundary point.

Next, we focus on the case $B(\mz^0)=0$ and derive a contradiction. Consider a cone
\begin{align}
    \mathcal{C}^\eta: = \{\mz \in R | \mz = \mz^0 + t e', t\in (0, \eta), \|e'\|= 1,\text{ and } (-e_1)^\top e' >1- \eta  \},
\end{align}
where $e_1 = (1, 0, 0, ..., 0)^\top$ and $0<\eta<1$ is a constant. To separate the coefficient of $G$ in $B$, define
\[
b(\mz):=\alpha_1-r-\gamma\sum_{j=1}^n\sigma_{1j}z_j u_{z_j}(\mz).
\]
Then
\[
B=bG+\sum_{i=1}^n z_i\left(\alpha_i-r-\gamma\sum_{j=1}^n\sigma_{ij}z_j u_{z_j}\right)G_{z_i}.
\]
In particular, $B=(1-\mu_1)b$ on $\mathbf S_1\cap R$.
We have the following lemma.
\begin{lemma}\label{lem open corn}
 When $\eta$ is small enough, we have that   $\mathcal{C}^\eta\cap {\bf S}_1 = \emptyset$ and $b(\mz)>0$ in $\mathcal{C}^\eta$.
\end{lemma}
\begin{proof}[Proof of Lemma \ref{lem open corn}]
\noindent\textit{(i) Positivity of $b$ in the cone.} Direct differentiation gives
\begin{align}
 b_{z_1}&=-\gamma\left(\sigma_{11}G+\sum_j\sigma_{1j}z_jG_{z_j}\right),\\
 b_{z_j}&=-\gamma\left(\sigma_{1j}u_{z_j}+\sum_k\sigma_{1k}z_k u_{z_kz_j}\right),\qquad j\ne1,\quad\text{a.e.}
\end{align}
By Lemmas \ref{lem weighted hessian bound} and \ref{lem u1 strong}, $b$ is locally Lipschitz and $b_{z_1}$ is continuous. At $\mz^0$, $G=1-\mu_1$ and $\nabla G=0$, so $B(\mz^0)=0$ gives $b(\mz^0)=0$ and
\begin{align}\label{equ nega par1}
 b_{z_1}(\mz^0)=-\gamma\sigma_{11}(1-\mu_1)<0.
\end{align}
Thus $b_{z_1}\leq-\gamma\sigma_{11}(1-\mu_1)/2$ near $\mz^0$. Let $C$ be a local Lipschitz constant of $b$. The restriction $t\mapsto b(\mz^0+t e')$ to the straight segment is absolutely continuous, and the bounds above imply
\begin{align}
 \frac{d}{dt}b(\mz^0+t e')\geq\frac{\gamma\sigma_{11}(1-\mu_1)}2(-e_1^\top e')-C|e'_{-1}|>0\qquad\text{a.e.}
\end{align}
for sufficiently small $\eta$, since $-e_1^\top e'>1-\eta$ and $|e'_{-1}|\leq\sqrt{2\eta}$. Integrating along the segment and using $b(\mz^0)=0$ gives $b(\mz^0+t e')>0$ for every $\mz^0+t e'\in\mathcal C^\eta$.

\noindent\textit{(ii) Exclusion of selling points.} On $\mathbf S_1\cap R$, $G=1-\mu_1$ and $\nabla G=0$, so $B=(1-\mu_1)b$ and $b_{z_1}=-\gamma\sigma_{11}(1-\mu_1)<0$. For fixed $z_{-1}$, the endpoint $(l_1^+(z_{-1}),z_{-1})$ lies in $R$ by Proposition \ref{prop transf}; the first part of the proof of Lemma \ref{lem B nega} gives $B\leq0$ there, hence $b\leq0$. Integrating $b_{z_1}$ from this endpoint to any selling point gives $b\leq0$ on $\mathbf S_1\cap R$. This contradicts (i) at any point of $\mathcal C^\eta\cap\mathbf S_1$, proving that this intersection is empty.
\end{proof}

Choose one point $\mz' = \mz^0-\epsilon e_1 $ in $\mathcal{C}^\eta \cap \bf N_1$,
the function $u$ is locally $C^\infty$ at $\mz'$.
Assume $$\nabla b = p\text{ at }\mz'.$$
From \eqref{equ nega par1}, $\|p\| > \frac{\gamma}{2} \sigma_{11} (1-\mu_1)$ when $\epsilon$ is small enough. Choose any unit direction $\mn$ such that $\frac{\mn^\top p}{\|p\|} > \sqrt{\epsilon}$, we have
\begin{align}
    &\frac{1}{\delta} (b(\mz^0 +\delta \mn ) - b(\mz^0) ) \\
   = &  \frac{1}{\delta} \big(b(\mz' +\delta \mn ) - b(\mz') )+ \frac{1}{\delta}\int_{-\epsilon}^0 (b_{z_1}(\mz^0+s e_1+\delta \mn) - b_{z_1}(\mz^0+s e_1) \big) ds\\
   \geq  & \mn^\top p + O(\delta)-\epsilon \|\nabla b_{z_1} \|_\infty\\
   = & \mn^\top p + O(\delta) - \epsilon C \|\nabla^2 G \|_\infty\\
   \geq & \sqrt{\epsilon} \|p\|+ O(\delta) -O(\epsilon)
\end{align}
First fix $\epsilon>0$ so small that the $O(\epsilon)$ error is at most $\frac14\sqrt{\epsilon}\|p\|$ in absolute value. For this fixed $\epsilon$, choose $0<\rho_\epsilon\leq\epsilon$ so that the $O(\delta)$ error, whose constant may depend on $\epsilon$, is at most $\frac14\sqrt{\epsilon}\|p\|$ in absolute value for every $0<\delta<\rho_\epsilon$, uniformly over the above unit directions. The expression is then positive for every $0<\delta<\rho_\epsilon$. That is, in the cone
\begin{align}
     \mathcal{C}^\epsilon:= \{\mz \in R | \mz = \mz^0 + t e', t\in (0, \rho_\epsilon), \|e'\|= 1,\text{ and }\frac{(e')^\top p}{\|p\|} > \sqrt{\epsilon} \},
\end{align}
we have $b>0$. Since $b\leq0$ on $\mathbf S_1\cap R$ by part (ii) of the proof of Lemma \ref{lem open corn}, this cone contains no selling points, so $G>1-\mu_1$ and the second-order equation holds there. Thus,
\begin{align}
     & \frac{1}{2} \sum_{i, j} \sigma_{ij} z_i z_j G_{z_iz_j} + \sum_{i}\sigma_{i1} z_i G_{z_i} +\sum_{i} (\alpha_i -r) z_i G_{z_i} \\
    & - \gamma \Big(\sum_{i, j} \sigma_{ij} z_i z_j u_{z_i} G_{z_j}\Big) = -b G\leq 0\text{ in }\mathcal{C}^\epsilon
\end{align}
and $G \geq 1-\mu_1$ on $\partial \mathcal{C}^\epsilon$. By Hopf's lemma, we have $\limsup_{t\to 0+} \frac{G(\mz^0-te_1) - G(\mz^0)}{t^\beta} >0$ for some constant $\beta>1$.
When $\epsilon$ is small, $\beta$ can be close to $1$ (see \cite{miller1971extremal} and \cite{kriventsov2024asymptotic}). That contradicts the $W^2_\infty$ regularity of $G$ and $G = 1-\mu_1$, $G_{z_1} = 0$ at $\mz^0$.
\end{proof}

\section{Classical solution for two-dimensional problems}\label{sec classical}
Theorem~\ref{thm local lipschitz selling} gives local $C^{2,\alpha}$ regularity at non-corner free-boundary points. We therefore only need to consider the corner points. We first prove Lipschitz regularity of the free boundaries at these points. In this section, write $\mathbf{NN}:=\mathbf{N}_1\cap\mathbf{N}_2$, $\mathbf{SN}:=\mathbf{S}_1\cap\mathbf{N}_2$, $\mathbf{NB}:=\mathbf{N}_1\cap\mathbf{B}_2$, and $\mathbf{SB}:=\mathbf{S}_1\cap\mathbf{B}_2$.

\begin{theorem}\label{thm corner lipschitz}
Let $n=2$ and suppose that Assumption~\ref{ass positive no trading} holds. Then the free boundaries $l_i^\pm$, $i=1,2$, are globally Lipschitz continuous.
\end{theorem}
\begin{proof}
To show this, we only need to prove that, for some constant $C>0$,
\begin{equation}\label{equ corner derivative comparison}
|u_{z_1z_2}|<-C u_{z_1z_1},\qquad |u_{z_1z_2}|<-C u_{z_2z_2}\quad\text{in }\mathbf{NN}.
\end{equation}

Without loss of generality, we consider $C u_{z_1z_1}\pm u_{z_1z_2}$, and will prove it to be negative.
In $\bf NN$, $G = u_{z_1}$ satisfies
$0 = \mathcal L_2 G + B$ where $\mathcal L_2$ and $B$ are defined in \eqref{equ G ope NN}.
Take the transform $z_1=e^w$, $z_2=e^v$, and write $G(w,v):=u_{z_1}(e^w,e^v)$. In the following calculations, $u_v$ and $u_{vv}$ denote derivatives of the transformed function $(w,v)\mapsto u(e^w,e^v)$. Then this PDE turns into
\begin{align}
0=&\frac12\Big(\sigma_1^2G_{ww}
+\sigma_2^2G_{vv}
+2\rho\sigma_1\sigma_2G_{wv}\Big)\\
&+\left(\alpha_1-r+ \frac12\sigma_1^2\right)G_w
+ \left(\alpha_2-r+\rho\sigma_1\sigma_2-\frac12\sigma_2^2\right)G_v +(\alpha_1-r)G \\
&-\gamma
\Big[
\sigma_1^2 e^w\left(G^2+GG_w\right)
+\sigma_2^2 u_v G_v
+\rho\sigma_1\sigma_2
\left(
u_v G
+u_v G_w
+e^w GG_v
\right)
\Big].
\end{align}
Therefore, define $H = G_w$, and we derive by taking $\partial_w$ on both sides that
\begin{align}
0={}&
\frac12\Big(
\sigma_1^2H_{ww}
+\sigma_2^2H_{vv}
+2\rho\sigma_1\sigma_2H_{wv}
\Big)\\
&
+\left(\alpha_1-r+ \frac12\sigma_1^2\right)H_w
+
\left(\alpha_2 -r + \rho\sigma_1\sigma_2-\frac12\sigma_2^2\right)H_v
+(\alpha_1-r)H \notag \\
&-\gamma
\Bigg[
\sigma_1^2 e^w
\left(
G^2+3GH+H^2+GH_w
\right)
+\sigma_2^2
\left(
e^w G^2_v+u_vH_v
\right) \notag \\
&\qquad\qquad
+\rho\sigma_1\sigma_2
\Big(
e^w G_v(G+H)
+u_v(H+H_w)
+e^w\left(GG_v+HG_v+GH_v\right)
\Big)
\Bigg]. \label{equ H Gw}
\end{align}
Similarly, define $J = G_v$, then by taking $\partial_v$ on both sides, we have
\begin{align}
0={}&
\frac12\Big(
\sigma_1^2J_{ww}
+\sigma_2^2J_{vv}
+2\rho\sigma_1\sigma_2J_{wv}
\Big)\\
&
+\left(\alpha_1-r+ \frac12\sigma_1^2\right)J_w
+
\left(\alpha_2-r+\rho\sigma_1\sigma_2-\frac12\sigma_2^2\right)J_v
+(\alpha_1-r)J \notag \\
&-\gamma
\Bigg[
\sigma_1^2e^w
\left(
2GJ+JH+GJ_w
\right)
+\sigma_2^2
\left(
u_{vv}J+u_vJ_v
\right) \notag \\
&\qquad\qquad
+\rho\sigma_1\sigma_2
\Big(
u_{vv}(G+G_w)
+u_v(J+J_w)
+e^w\left(J^2+GJ_v\right)
\Big)
\Bigg].\label{equ J Gv}
\end{align}
Therefore, define $L: = C H \pm J = C G_w \pm G_v = (C z_1) u_{z_1z_1}\pm z_2 u_{z_1z_2}$, we only need to prove that $L <0$ in $\Omega_N$ by contradiction, where
$$\Omega_N:= \{ (w, v) \in \mathbb{R}^2| (e^w, e^v) \in {\bf NN}   \}$$
is the corresponding open set ${\bf NN}$ under the new coordinates.
By Assumption~\ref{ass positive no trading}, $1/M\leq z_1,z_2\leq M$ in $\mathbf{NN}$. Increasing $M$ if necessary, we also have $\overline{\mathbf{NN}}\subset(1/M,M)^2$. Thus $\Omega_N$ is bounded and these inequalities imply the targeted estimates after increasing $C$. We will prove $L<0$ by the comparison principle.

{\bf Step 1}. We first derive the equation satisfied by the function $L$.

From \eqref{equ H Gw} and \eqref{equ J Gv} we have
\begin{align}
   0 =  & \frac{1}{2} \Big(\sigma_1^2 {L}_{ww} +  \sigma_2^2  {L}_{vv} + 2 \rho \sigma_1 \sigma_2 {L}_{wv} \Big)\\
   & +\left(\alpha_1-r+ \frac{1}{2} \sigma_1^2 \right){L}_w +\left(\alpha_2-r+\rho\sigma_1\sigma_2-\frac12\sigma_2^2\right)L_v
    +(\alpha_1-r)L \label{equ L 0} \\
   & - \gamma \bigg[ \sigma_1^2 e^w \Big( 2 GL + HL + G L_w \Big) + \sigma_2^2 (u_v L_v ) +\rho\sigma_1 \sigma_2 \Big(u_v (L+ L_w) + e^w (J L + G L_v )\Big)\bigg]\\
   & - C \gamma e^w \bigg[ \sigma_1^2(G^2 + G H) + \sigma_2^2(G_v^2 ) + \rho \sigma_1\sigma_2 (2 G G_v + G_v H)\bigg]\label{equ L 2} \\
   & \mp \gamma\bigg( \sigma_2^2 u_{vv} J + \rho \sigma_1\sigma_2 u_{vv} (G+G_w)  \bigg).
   \label{equ L 3}
\end{align}
On the one hand, Lemma~\ref{lem weighted hessian bound} and $z_i\geq1/M$ show that $\nabla^2u$ is uniformly bounded in $\mathbf{NN}$. By noticing that
$$G = u_{z_1},\ G_w = z_1 u_{z_1z_1},\ u_v = z_2 u_{z_2} = e^v u_{z_2},\ u_{vv} = z^2_2 u_{z_2z_2} + z_2 u_{z_2} $$
and $z_1, z_2$ are bounded, we see that
\eqref{equ L 3} is bounded by a constant independent of $C$.

On the other hand, we estimate \eqref{equ L 2}. Noticing that $G\geq 1-\mu_1$ and
\begin{align}\label{equ esti nabla G}
\|\nabla G\|_\infty \leq M \|\nabla^2 u\|_\infty< \infty,
\end{align}
we have
\begin{align}
      & - C\Big[ \sigma_1^2(G^2 + G H) + \sigma_2^2(G_v^2 ) + \rho \sigma_1\sigma_2 (2 G G_v + G_v H)\Big]\\
     = & - C\Big[ \sigma_1^2 G^2  + \sigma_2^2G_v^2  + 2\rho \sigma_1\sigma_2  G G_v \Big] -\Big[(\sigma_1^2 G +\rho \sigma_1\sigma_2 G_v) (L \mp J) \Big] \\
     = & - C\Big[ \sigma_1^2 G^2  + \sigma_2^2G_v^2  + 2\rho \sigma_1\sigma_2  G G_v \Big] \pm (\sigma_1^2 G +\rho \sigma_1\sigma_2 G_v) G_v
     -(\sigma_1^2 G +\rho \sigma_1\sigma_2 G_v) L \\
     \leq &-C \delta (G^2 +G_v^2) - (\sigma_1^2 G +\rho \sigma_1\sigma_2 G_v) L, \label{equ esti L2}
\end{align}
Here $\delta>0$ depends only on the positive definite covariance matrix. For sufficiently large $C$, its quadratic form absorbs the term $\pm(\sigma_1^2G+\rho\sigma_1\sigma_2G_v)G_v$, giving the last inequality with $\delta$ independent of $C$.

Noticing that
$$ G^2 +G_v^2 \geq \pm G G_v + \frac{1}{2} G^2 + \frac{1}{2} G_v^2 \geq \pm G G_v + \frac{1}{2} G^2, $$
we see
\begin{align}
    \eqref{equ esti L2}
    \leq & \mp C\delta G G_v - \frac{C \delta}{2} G^2  - (\sigma_1^2 G +\rho \sigma_1\sigma_2 G_v) L\\
    = & - C \delta G (L - C G_w ) - \frac{C \delta}{2} G^2  - (\sigma_1^2 G +\rho \sigma_1\sigma_2 G_v) L\\
    = & -\Big(C \delta G +  (\sigma_1^2 G +\rho \sigma_1\sigma_2 G_v) \Big) L  + C^2 \delta G G_w - \frac{C \delta}{2} G^2
\end{align}
Due to \eqref{equ esti nabla G},
and $G_w=e^w u_{z_1z_1}\leq0$, $G\geq1-\mu_1>0$, multiplying the last bound by $\gamma e^w$ and substituting it into \eqref{equ L 0}--\eqref{equ L 3} gives, for sufficiently large $C>0$,
\begin{align}
  \frac{1}{2} \Big(\sigma_1^2{L}_{ww} + \sigma_2^2{L}_{vv} + 2\rho\sigma_1\sigma_2{L}_{wv}\Big) + b_1(w,v)L_w+b_2(w,v)L_v-h(w,v)L>0.\label{equ L CP}
\end{align}
Here $b_1,b_2$ are bounded, and the zeroth-order coefficient has the form $-h=-C\gamma e^w\delta G+R(w,v)$, with $R$ bounded independently of $C$. Since $e^wG\geq(1-\mu_1)/M$, we have $h\geq0$ for large $C$. The negative term $-C\gamma e^w\delta G^2/2$ absorbs the bounded term \eqref{equ L 3}, while $C^2\gamma e^w\delta GG_w\leq0$, which gives the strict inequality. Thus a positive maximum cannot be achieved in the interior of the no-action region.

{\bf Step 2}. We then prove $L< 0$ in $\Omega_N$.

We only need to show $L\leq 0$, then due to the strong maximum principle, $L<0$ in $\Omega_N$.
We prove by contradiction.
If $\sup_{\Omega_N}L=2d>0$, choose a sequence $(w_i,v_i)\in\Omega_N$ with $L(w_i,v_i)\to2d$ and $L(w_i,v_i)\geq d$. Since $\Omega_N$ is bounded, after passing to a subsequence, $(w_i,v_i)\to(w_0,v_0)\in\overline{\Omega}_N$.

{\bf Case 1}. If  $(w_0, v_0)\in {\Omega}_N$. That is, at $P := (e^{w_0}, e^{v_0})$, $u_{z_i}\in (1-\mu_i, 1+\lambda_i)$, $i = 1, 2$.  Then \eqref{equ L CP} implies a contradiction.

{\bf Case 2}. The point $(w_0,v_0)$ is on the trading boundary, but is not a corner. At $P:=(e^{w_0},e^{v_0})$, one derivative $u_{z_i}$ equals $1-\mu_i$ or $1+\lambda_i$, while $u_{z_j}\in(1-\mu_j,1+\lambda_j)$, $\{i,j\}=\{1,2\}$. By Theorem~\ref{thm local lipschitz selling}, $u_{z_1z_2}$ is continuous around $P$ and $u_{z_1z_2}(P)=0$, since the binding derivative has a local extremum there. As $u_{z_1z_1}\leq0$, we have $\limsup_i L(w_i,v_i)\leq0$, a contradiction.

{\bf Case 3}. The point $(w_0, v_0)$ is at a corner. Without loss of generality, assume $u_{z_1} = 1-\mu_1$ and $u_{z_2} = 1+\lambda_2$.

By definition,
\begin{align}
\theta - \mathcal{L} u +\mathcal{A}(u) = 0
\end{align}
at the points $z^i:=(e^{w_i},e^{v_i})$, $i\geq1$.

We have the following lemma.
\begin{lemma}\label{lem strict nega nabla2}
   The quantities $(\sigma_1^2z_1^2u_{z_1z_1}+\sigma_2^2z_2^2u_{z_2z_2}+2\rho\sigma_1\sigma_2z_1z_2u_{z_1z_2})|_{z^i}$ have a uniform negative upper bound $-\epsilon<0$.
\end{lemma}
\begin{proof}
  By concavity, $G_w(w_i,v_i)=e^{w_i}u_{z_1z_1}(z^i)\leq0$. Thus $L=CG_w\pm G_v\geq d$ gives $\pm G_v(w_i,v_i)\geq d$, and hence $|u_{z_1z_2}(z^i)|=e^{-v_i}|G_v(w_i,v_i)|\geq d/M>0$. Again by concavity, at $z^i$,
  \begin{align}
      &\sigma_1^2z_1^2u_{z_1z_1}+\sigma_2^2z_2^2u_{z_2z_2}+2\rho\sigma_1\sigma_2z_1z_2u_{z_1z_2}\\
\leq & -2\sigma_1\sigma_2z_1z_2\sqrt{u_{z_1z_1}u_{z_2z_2}}+2\rho\sigma_1\sigma_2z_1z_2u_{z_1z_2}\\
\leq & -2(1-|\rho|)\sigma_1\sigma_2z_1z_2|u_{z_1z_2}|\\
\leq & -\epsilon
  \end{align}
  for some $\epsilon>0$.
\end{proof}
Set
\[
\Phi(\mz):=\theta-\sum_{i=1}^2(\alpha_i-r)z_i u_{z_i}(\mz)+\mathcal A(u)(\mz).
\]
This function is locally Lipschitz by Lemma~\ref{lem weighted hessian bound}. The equation in $\mathbf{NN}$ and Lemma~\ref{lem strict nega nabla2} give
\[
\begin{aligned}
\Phi(P)&=\lim_{i\to\infty}\Phi(z^i)\\
&=\lim_{i\to\infty}\frac12\left(\sigma_1^2z_1^2u_{z_1z_1}+\sigma_2^2z_2^2u_{z_2z_2}+2\rho\sigma_1\sigma_2z_1z_2u_{z_1z_2}\right)\Big|_{\mz=z^i}
\leq-\frac{\epsilon}{2}.
\end{aligned}
\]
On the other hand, choose $Q=P+(\eta,-\eta)$ for small $\eta>0$. Concavity and the gradient constraints make $u$ affine on the trading quadrant from $P$, so $Q$ is in the interior of $\mathbf{SB}$ and $D^2u(Q)=0$. The variational inequality \eqref{equ PDE u} gives $\Phi(Q)\geq0$. Letting $\eta\downarrow0$ yields $\Phi(P)\geq0$, a contradiction. The other corner types use their corresponding trading directions.

\end{proof}

We will prove that the second-order derivatives satisfy $|u_{z_1z_1}(z_1, z_2)|,|u_{z_1z_2}(z_1, z_2)|, |u_{z_2z_2}(z_1, z_2)| \leq C \|(z_1, z_2) -P\|^\alpha$ for $(z_1, z_2)\in B_\epsilon(P)$ for some small $\epsilon>0$ and the $\bf SB$ region corner point $P = (P_1, P_2)$.

We will establish $C^{2,\alpha}$ regularity near $P$ using the Schauder estimate.
\begin{figure}
\begin{minipage}{\linewidth}\centering		\includegraphics[width=0.8\textwidth]{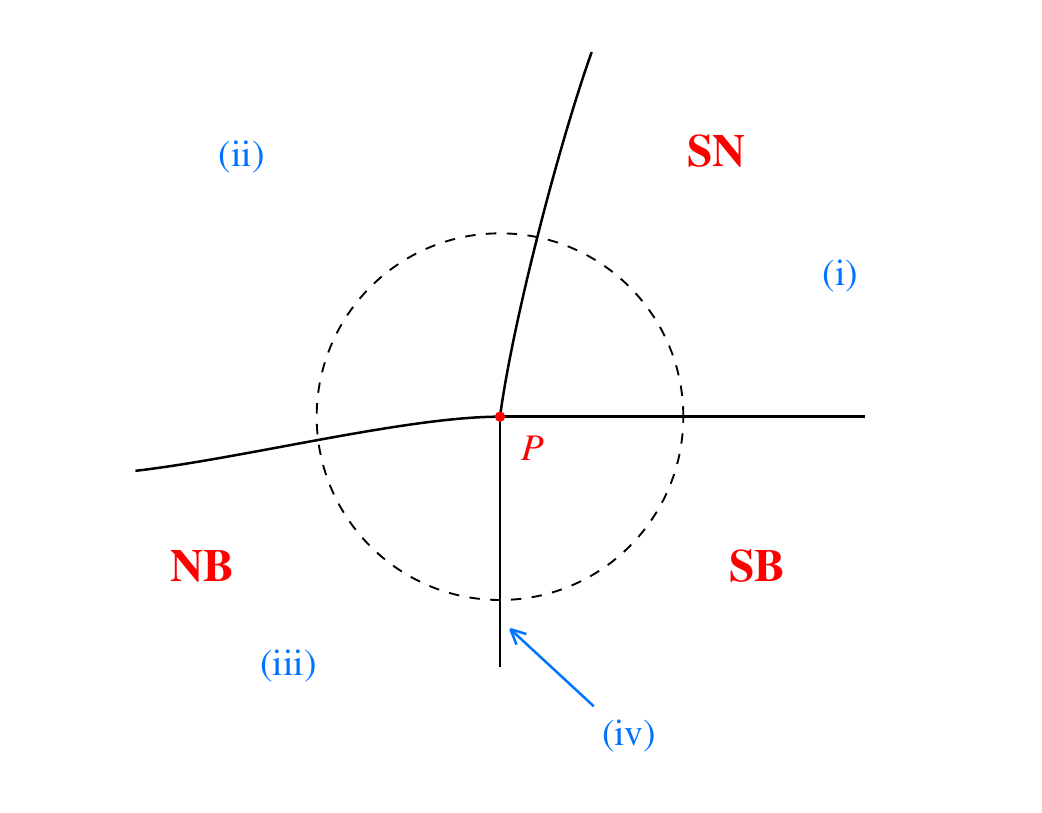}
\end{minipage}
\caption{Illustrative figure of a corner point}
\end{figure}
Here $P\in\overline{\mathbf{NN}}$, $u_{z_1}(P)=1-\mu_1$, and $u_{z_2}(P)=1+\lambda_2$. Choose $\epsilon$ so that $B_\epsilon(P)$ lies in the positive quadrant and only the phases $\mathbf{NN}$, $\mathbf{SN}$, $\mathbf{NB}$, and $\mathbf{SB}$ occur.

Using $\Phi$ from the proof of Theorem~\ref{thm corner lipschitz}, define
\[
\widehat \Phi(z_1,z_2)=\begin{cases}
\Phi(z_1,z_2),&\mathbf{NN},\\
\Phi(l_1^+(z_2),z_2),&\mathbf{SN},\\
\Phi(z_1,l_2^-(z_1)),&\mathbf{NB},\\
0,&\mathbf{SB}.
\end{cases}
\]

\begin{lemma}\label{lem corner equation}
For sufficiently small $\epsilon>0$, the function $\widehat\Phi$ is Lipschitz on $B_\epsilon(P)$, and $u$ satisfies \eqref{equ corner global equation} almost everywhere in $B_\epsilon(P)$.
\end{lemma}
\begin{proof}[{Proof}]
We first prove that $\widehat \Phi$ is Lipschitz continuous.

We begin by proving $\Phi(P) = 0$. Consider $Q \in \mathbf{NN}\cap B_\epsilon(P)$ and $Q\neq P$. We see that
\begin{align}\label{equ Q NN}
    \theta -\mathcal{L} u + \mathcal{A}(u)|_{Q} = 0.
\end{align}
For any sufficiently small $\eta$, $Q' = P + (\eta, -\eta)$ is in the interior of ${\bf SB}$; hence we have $\nabla^2 u|_{Q'} = 0$ with
\begin{align}\label{equ Q' Sb}
   \theta -\mathcal{L} u + \mathcal{A} (u) |_{Q'} \geq 0.
\end{align}
Indeed, the gradient constraints and concavity make $u$ affine on the trading quadrant $P+\{(a,-b):a,b\geq0\}$.
By concavity and \eqref{equ Q NN}, $\Phi(Q)\leq0$. Letting $Q\to P$ within $\mathbf{NN}$ gives $\Phi(P)\leq0$. Since $D^2u(Q')=0$, \eqref{equ Q' Sb} gives $\Phi(Q')\geq0$; letting $\eta\downarrow0$ yields $\Phi(P)\geq0$. Thus $\Phi(P)=0$.
Since $u\in C^{1,1}$ and the free boundaries are Lipschitz,
each branch of $\widehat\Phi$ is Lipschitz.

We next verify the equation for $u$.
Choose one point $Q = (z_1, z_2)$ in $B_\epsilon(P)$.

If this point $Q$ is in the interior of the sell region of stock 1, i.e., $Q\in (\mathbf{SB}\cup\mathbf{SN})^o$ locally, then immediately we have $u_{z_1z_1} = 0$ from the variational inequality. Here also $u_{z_1z_2}=0$, since $u_{z_1}=1-\mu_1$ is constant.

Suppose $Q$ is in the interior of the $\bf NB$ region. Then we see from the buy condition $u_{z_2}=1+\lambda_2$ that
$$u_{z_1z_1}(z_1,z_2)=u_{z_1z_1}\big(z_1,l_2^-(z_1)\big).$$
The projection point is a non-corner point of $\overline{\mathbf{NN}}$, as shown above. Theorem~\ref{thm local lipschitz selling} gives $C^{2,\alpha}$ regularity there and justifies the identity. Also $u_{z_2z_2}=u_{z_1z_2}=0$ in $\mathbf{NB}^o$. Similarly, in $\mathbf{SN}^o$,
\[
u_{z_2z_2}(z_1,z_2)=u_{z_2z_2}\big(l_1^+(z_2),z_2\big),
\]
while $D^2u=0$ in $\mathbf{SB}^o$.

Thus the derivatives involving an active trading direction vanish, while each surviving second derivative and its coefficient are unchanged by the corresponding projection. Together with \eqref{equ Q NN} and the fact that the free-boundary graphs have measure zero, this gives
\begin{equation}\label{equ corner global equation}
\frac12\sum_{i,j=1}^2\sigma_{ij}z_iz_j u_{z_iz_j}=\widehat \Phi
\quad\text{almost everywhere in }B_\epsilon(P).
\end{equation}
\end{proof}

Lemma~\ref{lem weighted hessian bound} gives $u\in C^{1,1}$ locally, and the operator in \eqref{equ corner global equation} has smooth coefficients and is uniformly elliptic on this ball. Applying the interior Schauder estimate to the equation in Lemma~\ref{lem corner equation}, we obtain, for every $0<\alpha<1$,
\[
\|u\|_{C^{2,\alpha}(B_{\epsilon/2}(P))}
\leq C_\alpha\left(\|u\|_{L^\infty(B_\epsilon(P))}
+\|\widehat\Phi\|_{C^{0,\alpha}(B_\epsilon(P))}\right);
\]
see \cite[Chapter~6]{gilbarg1998elliptic}.

\bibliographystyle{plain}
\bibliography{refe_cost}

\end{document}